\documentclass[a4paper,reqno,11pt]{amsart}
\usepackage{amsfonts,amsmath,amsthm,amssymb,stmaryrd}
\usepackage{hyperref}
\usepackage{color}

\newtheorem{Theorem}{Theorem}[section]
\newtheorem{Lemma}[Theorem]{Lemma}

\numberwithin{equation}{section} \allowdisplaybreaks
\allowdisplaybreaks 
\begin{document}
\author{Yi Peng}
\address{College of Mathematics and Statistics, Chongqing University, Chongqing, 401331,  China.}
\email[Y. Peng]{20170602018t@cqu.edu.cn}

\author{Huaqiao Wang}
\address{College of Mathematics and Statistics, Chongqing University, Chongqing, 401331,  China.}
\email[H.Q. Wang]{wanghuaqiao@cqu.edu.cn}

\author{Qiuju Xu}
\address{School of Mathematical Sciences, Chongqing Normal University, Chongqing, 400047, China.}
\email[Q.J. Xu]{smlynice@163.com}

\title[Derivation of Hall-MHD equations]
{Derivation of the Hall-MHD equations from the Navier-Stokes-Maxwell equations}
\thanks{Corresponding author: wanghuaqiao@cqu.edu.cn}
\

\begin{abstract}
By using a set of scaling limits, the authors in \cite{ADFL,SS} proposed a framework of deriving the Hall-MHD equations from the two-fluids Euler-Maxwell equations for electrons and ions. In this paper, we derive the Hall-MHD equations from the Navier-Stokes-Maxwell equations with generalized Ohm's law in a mathematically rigorous way via the spectral analysis and energy methods.
\end{abstract}

\maketitle
\section{Introduction}

The Hall-MHD equations have the following form in $[0,T]\times\mathbb{R}^{3}$:
\begin{equation}\label{hallmhd}
\begin{cases}
\partial _tu+{\rm div}(u\otimes u)-\Delta u+\nabla p= (\nabla\times B)\times B,&{\rm div}u = 0, \\
\partial _tB+ \nabla\times((\nabla\times B)\times B)-\nabla\times (u\times B)=\Delta B,&{\rm div}B=0.
\end{cases}
\end{equation}
Here, $u(t,x)$ and $B(t,x)$ are the fluid velocity and magnetic field, and $p$ is the pressure.
The equations \eqref{hallmhd} have been studied in physics for a long time and have many applications in the field of physics such as earth generators, etc. In general, the Hall term $\nabla\times((\nabla\times B)\times B)$ does not affect the properties of the magnetohydrodynamic equations. Under strong magnetic field or when the plasma density is relatively small, such as the magnetic field reconnection in space plasmas \cite{FTG,HG}, star formation \cite{BT,WM}, geo-dynamo \cite{MGM} and so on, the Hall effect can not be ignored. Lighthill \cite{LMJ} first studied the Hall effects and physically derived the Hall terms. Next, we sketch the formal derivation in \cite{LMJ}. Since the fluids are composed of charged particles, their flow generates electric field $E$ and magnetic field $B$. Here the magnetic field $B$ satisfies Ampere's law:
\begin{equation}\label{con:alaw}
j=\nabla\times B,
\end{equation}
where $j$ is the current density, while the electric field $E$ satisfies Faraday's law:
\begin{equation}\label{con:flaw}
\partial_{t}B=-\nabla\times E.
\end{equation}
If $n_{e}$ and $n_{i}$ are the number densities of ions and electrons, we have
$$\rho=n_{i}m_{i}+n_{e}m_{e},\;\; u=\frac{n_{i}m_{i}u_{i}+n_{e}m_{e}u_{e}}{n_{i}m_{i}+n_{e}m_{e}},$$
where $m_{i}$, $m_{e}$ are the masses of ion and electron, $u_{i}$ and $u_{e}$ denote the velocities, $\rho$ stands for the charge density. Since $m_{i}$ is much larger than $m_{e}$ (more than 1800 times), $u$ is very close to $u_{i}$,  $j$ can be approximated as
\begin{equation}\label{con:jandve}
j= en_{e}(u-u_{e}),
\end{equation}
where $-en_{e}$ is the electron charge density. The electrons and ions momentum equations have the following form:
\begin{equation}\label{con:momentumofelectron}
n_{e}m_{e}(\partial_{t}u_{e}+u_{e}\cdot\nabla u_{e})=-\nabla p_{e}-M-n_{e}e(E+u_{e}\times B),
\end{equation}
\begin{equation}\label{con:momentumofion}
n_{i}m_{i}(\partial_{t}u_{i}+u_{i}\cdot\nabla u_{i})=-\nabla p_{i}+M+n_{i}Ze(E+u_{i}\times B),
\end{equation}
where $Z$ is the average charge of the ions, $p_{i}$ and $p_{e}$ are the pressures of ions and electrons, $M=-n_{e}e\kappa j$ denotes the momentum loss of electrons per unit volume during the collision and $\kappa$ represents the resistivity.
Since the left-hand side of \eqref{con:momentumofelectron} is very small compared to the left-hand side of \eqref{con:momentumofion}, by putting the left-hand side of \eqref{con:momentumofelectron} equal to zero, we can get the following approximation:
\begin{equation}\label{E}
E=\kappa j+\frac{\nabla p_{e}}{n_{e}e}-u_{e}\times B.
\end{equation}
Combining \eqref{con:alaw}, \eqref{con:flaw}, \eqref{con:jandve} and \eqref{E}, we can formally derive the second equation of \eqref{hallmhd} when $en_{e}=1$ and $\kappa=1$ (see \cite{CLMBC,LMJ} and the references therein for more details).

However, the above formal derivation about Hall term may not be satisfactory. Since, in general, Ampere's law with Maxwell's correction
$$j+\partial_{t}E=\nabla\times B$$
is required in electromagnetic effects.

Notice that from the two-phase flow model or dynamic model, the Hall-MHD system is derived through the two-layer scale limit (scaling limits) in reference \cite{ADFL,SS}. However, in these literature, only a framework is provided and rigorous mathematical proof is not given. In this paper, we want to rigorously derive Hall-MHD system \eqref{hallmhd} (a dimensionless version) from the incompressible Navier-Stokes-Maxwell equations with generalized Ohm's law, which takes into account the full electromagnetic phenomenon described by the Maxwell system. We first consider the following scaled two-fluid incompressible Navier-Stokes-Maxwell system stemming from \cite{ADFL} with parameters $\varepsilon$ and $\gamma$:
\begin{equation}\label{twofiuld}
\begin{cases}
\varepsilon ^{2}\left(\partial _tu_e+{\rm div}\left(u_e\otimes u_e\right)-\Delta u_{e}\right)+\nabla p_{e}=-\alpha^{2}(E+u_e\times B)-\beta(u_e-u_i),
&{\rm div}u_e=0, \\
\partial _tu_i+{\rm div}(u_i\otimes u_i)-\Delta u_{i}+\nabla p_{i}=\alpha^{2}(E+u_i\times B)-\beta(u_i-u_e),&{\rm div}u_i=0, \\
\gamma^{2}\partial _tE-\nabla \times B=-j,\\
\partial _tB+\nabla \times E=0,&{\rm div}B=0,\\
j=\frac{1}{\eta}(u_i-u_e),
\end{cases}
\end{equation}
where $0<\varepsilon^{2}\ll 1$ denotes the mass ratio of the electron to the ion, $\alpha$ denotes the ratio of the electric energy to the thermal energy, $\beta$ stands for the
relaxation frequency of the electron and ion velocities due to collisions. $\eta$ denotes the ratio of the charge current scale to the electron or
ion current scales. $\gamma>0$ is the ratio of the fluid velocity to the speed of light.

Next, we will give a formal asymptotic analysis of equations \eqref{twofiuld}. For fixed $\gamma$, setting $\alpha^{2}\eta=1$, and letting $u_{i}=u$, $p=p_{i}+p_{e}$, and  $\varepsilon\rightarrow 0$,  we can formally get the following incompressible Navier-Stokes-Maxwell system with generalized Ohm's law:
\begin{equation}\label{con:NSM}
\begin{cases}
\frac{1}{\eta}(E+u\times B)-\beta\eta j+\nabla p_{e}=j\times B,&{\rm div}j = 0,\\
\partial _tu+{\rm div}(u\otimes u)-\Delta u+\nabla p=j\times B,&{\rm div}u = 0, \\
\gamma^{2}\partial _tE-\nabla \times B=-j,\\
\partial _tB+\nabla \times E=0,&{\rm div}B=0,
\end{cases}
\end{equation}
which enjoys the formal energy conservation law:
\begin{equation}\label{energy}
\frac{1}{2}\frac{d}{dt}\left(\|u\|_{L_{x}^{2}}^{2}+\|B\|_{L_{x}^{2}}^{2}+\|\gamma E\|_{L_{x}^{2}}^{2}\right)+\left\|\nabla u\right\|_{L_{x}^{2}}^{2}+\beta\eta^{2}\|j\|_{L_{x}^{2}}^{2}=0.
\end{equation}

Furthermore, letting $\gamma\rightarrow 0$, and then cancelling out the terms $E$ and $j$, the Hall-MHD system \eqref{hallmhd} (a dimensionless version) is obtained formally.
We know that for any initial data $(u^{0},B^{0})\in L^{2}(\mathbb{R}^{3})$ such that ${\rm div}u^{0}=0$, there exists a global weak solution $(u,B)\in L^{\infty}(\mathbb{R}_{+},L^{2}(\mathbb{R}^{3}))\bigcap
L^{2}(\mathbb{R}_{+},H^{1}(\mathbb{R}^{3}))$ to the system \eqref{hallmhd} (for example, see \cite{CDL}).

Now, we begin to introduce the research history of the above systems. For the following Navier-Stokes-Maxwell system with ideal Ohm's law
\begin{equation}\label{NSM}
\begin{cases}
\partial _tu+{\rm div}(u\otimes u)-\Delta u=-\nabla p+j\times B,&{\rm div}u = 0, \\
\partial _tE-\nabla \times B=-j,&j=\sigma(E+u\times B),\\
\partial _tB+\nabla \times E=0,&{\rm div}B=0,
\end{cases}
\end{equation}
where $\sigma> 0$ is the electrical conductivity of the fluid, the existence of global and finite energy weak solutions remains an interesting open problem, in both the dimensions $d=2,3$, for lack of compactness in the magnetic field $B$, which prevents taking limits in the Lorentz force $j\times B$. However, existence results are available when more regularity on the initial data is imposed.
Masmoudi \cite{MN} built up the global existence and uniqueness of the strong solutions to \eqref{NSM} in two dimensions with any large initial data $(u_{0},E_{0},B_{0})\in L^{2}\times H^{s},~ s>0$.
Later, Ibrahim-Keraani \cite{IK} established the global existence of strong solutions to \eqref{NSM} for the small initial data $(u_{0},E_{0},B_{0})\in \dot{B}^{1/2}_{2,1}\times H^{1/2}$ in three dimensions, and the initial data $(u_{0},E_{0},B_{0})\in \dot{B}^{0}_{2,1}\times L^{2}_{{\rm log}}$ in two dimensions.
These results are extended in \cite{GI} to small initial data $(u_{0},E_{0},B_{0})\in H^{1/2}$ in three dimensions, and $(u_{0},E_{0},B_{0})\in L^{2}\times L^{2}_{{\rm log}}$ in two dimensions by applying fixed point arguments. Recently, Ars\'{e}nio-Gallagher \cite{AG} studied global existence of weak solutions in largest possible functional spaces based on a new maximal estimate on the heat equation in Besov spaces, under the hypothesis that the initial data $(u_{0},E_{0})$ lies in the energy space, and that the initial magnetic field $B_0$ is sufficient small in $H^{s}(\mathbb{R}^{3})$ with $s\in \left[\frac{1}{2},\frac{3}{2}\right)$. For the initial data without any restriction, they also built up the global existence of weak solutions to \eqref{NSM} in $\mathbb{R}^{2}$. In particular, the result in two-dimensional space improves the asymptotic limit established in \cite{AIM} and the global well-posedness result obtained in \cite{MN}.
See also \cite{IY,GISY} and the references therein for more well-posedness results of the system \eqref{NSM}.

There have been many research results for the Hall-MHD equations.  Acheritogaray-Degond-Frouvelle-Liu \cite{ADFL} proposed a framework of deriving the Hall-MHD equations from the two-fluids Euler-Maxwell system for electrons and ions (see also \cite{SS}) and proved the global existence of weak solutions in the periodic domain.
Chae-Degond-Liu \cite{CDL} obtained the global existence of weak solutions and the local well-posedness of classical solutions in the three-dimensional whole space.
They also established the blow-up criterion and the global existence of classical solutions for small initial data. Later, Chae-Lee \cite{CL}
generalized the results of \cite{CDL}. Chae-Wolf \cite{CW} investigated the  partial regularity of the weak solutions of the three-dimensional Hall-MHD equations
on the plane and obtained that the space-time Hausdorff dimension of the set of possible singularities for a weak solution is at most 2. Chae-Schonbek \cite{CS}
established the optimal time decay rate of weak solutions in the three-dimensional whole space.  Chae-Weng \cite{CW} considered
the singularity formation for the Hall-MHD equations without resistivity in $\mathbb{R}^3$. Dumas-Sueur \cite{DS} investigated energy conservation of
weak solutions to the three-dimensional Hall-MHD equations. Dai \cite{DM} established the regularity criterion of the 3D incompressible resistive viscous Hall-MHD equations.
Benvenutti-Ferreira \cite{BF} considered the stability of global large strong solutions. Very recently, Dai \cite{DM1} obtained the
non-uniqueness of the Leray-Hopf weak solution by using the convex integration scheme. For more results on the Hall-MHD equations, see \cite{DRG,FHN,PM,WZ}. For the generalized Hall-MHD equations, Dai-Liu \cite{DH} established the local well-posedness in the Besov space. Chae-Wan-Wu \cite{CWW} obtained the local well-posedness of the smooth solution of the Hall-MHD equations with fractional magnetic diffusion. We refer the readers to
\cite{DH1,JZ,PZ,WR,WZ1,WYT,YZ} and  the references therein for the well-posedness results and the regularity criteria of the generalized Hall-MHD equations.

Compared to the system \eqref{NSM}, there are very little researches have been done on \eqref{con:NSM}. The existence of weak solutions to \eqref{con:NSM} in energy space seems very challenging for lacking of weak stability of the Lorentz force $j\times B$ included both in generalized Ohm's law and the momentum conservation equation, which also reflects the difficulties we encounter in the asymptotic analysis of \eqref{con:NSM}. In the present paper, we impose the hypothesis of existence of the global and finite energy weak solutions to system \eqref{con:NSM}.

Before starting the main result of this paper, let us first introduce the notations and conventions used throughout this paper.

\textbf{Notations:}
\begin{enumerate}
  \item For any positive $A$ and $B$, we use the notation $A\lesssim B$ to mean that there exists a positive constant $C$ such that $A\leqslant CB$.
  \item The Fourier transform is defined by
  $$\mathcal{F}f(\xi)=\hat{f}(\xi):=\frac{1}{(2\pi)^{\frac{3}{2}}}\int_{\mathbb{R}^{3}}e^{-ix\cdot\xi}f(x)dx,$$
  and its inverse is defined by
  $$\mathcal{F}^{-1}g(x)=\check{g}(x):=\frac{1}{(2\pi)^{\frac{3}{2}}}\int_{\mathbb{R}^{3}}e^{ix\cdot\xi}g(\xi)d\xi.$$
  \item For every $p\in[1,\infty]$, we denote the norm in the Lebesgue space $L^{p}$ by $\left\|\cdot\right\|_{L^{p}}$.
  \item For any normed space $X$, we employ the notation $L^{p}\left([0,T],X\right)$ to denote the space of functions $f$ such that for almost all $t\in (0,T)$, $f(t)\in X$ and $\left\|f(t)\right\|_{X}\in L^{p}(0,T)$. We simply denote the notation $L^{p}\left([0,T],X\right)$ by $L^{p}_{T}X$.
  \item The homogeneous Sobolev space $\dot{H}^{s}(\mathbb{R}^{3})$, for any $s\in \mathbb{R}$, as the subspace of tempered distributions whose Fourier transform is locally integrable endowed with the norm
      $$\left\|f\right\|_{\dot{H}^{s}}=\left(\int_{\mathbb{R}^{3}}\left|\xi\right|^{2s}\left|\hat{f}(\xi)\right|^{2}d\xi\right)^{\frac{1}{2}}.$$
\end{enumerate}

Now, we are ready to state our main result in this paper.
\begin{Theorem}\label{Mian} Let $s\in \left(\frac{1}{2},1\right)$ be fixed. For any $\gamma>0$, we assume that $\left(u^{\gamma},E^{\gamma},B^{\gamma}\right)$ is the global and finite energy weak solution to the incompressible Navier-Stokes-Maxwell system with generalized Ohm's law \eqref{con:NSM} for some uniformly bounded initial data
$$\left(u^{0\gamma},E^{0\gamma},B^{0\gamma}\right)\in \left(H^{s}\times \left(L^{2}\right)^{2}\right)\left(\mathbb{R}^{3}\right),$$
with ${\rm div}u^{0\gamma}={\rm div}B^{0\gamma}=0$, and suppose that the initial data converges weakly in $H^{s}\times \left(L^{2}\right)^{2}$, as $\gamma\rightarrow 0$, to some quantity
$$\left(u^{0},E^{0},B^{0}\right)\in \left(H^{s}\times \left(L^{2}\right)^{2}\right)\left(\mathbb{R}^{3}\right),$$
with ${\rm div}u^{0}={\rm div}B^{0}=0.$
Besides, we also assume that
\begin{equation}\label{regularhp}
u^{\gamma}\in L_{t}^{2}\dot{H}_{x}^{1+s} \text{ are uniformly bounded},
\end{equation}
and that for any $\delta>0$,
\begin{equation}\label{higj}
\limsup\limits_{\gamma\rightarrow 0}\left\|j_{\gg}^{\gamma}\right\|_{L_{t,x,\rm loc}^{2}}=0,
\end{equation}
where
$$j_{\gg}^{\gamma}=\mathcal{F}^{-1}\left(\chi_{\left\{|\xi|\geq\phi\left(\frac{\gamma}{\delta}\right)\right\}}\mathcal{F}j^{\gamma}\right),\;\phi(\gamma)=\gamma^{\frac{2}{2s-3}}.$$
Then, taking $\gamma\rightarrow 0$, up to extraction of a subsequence, $(u^{\gamma},B^{\gamma})$ converges weakly to a global and finite energy weak solution
$(u,B)$ of the Hall-MHD equations \eqref{hallmhd} (a dimensionless version) with initial data $(u^{0},B^{0})$.
\end{Theorem}

Now, we sketch the strategy of proving this theorem and point out some of the main difficult and techniques involved in the process. The convergence of the linear terms in \eqref{con:NSM} is easy to handle with thanks to the formal energy conservation law. Furthermore, the dissipation on $u$ is clearly sufficient to establish the weak stability of the nonlinear terms $u\times B$ and ${\rm div}(u\otimes u)$. Thus, in order to prove the asymptotic limit of \eqref{con:NSM} rigorously, it is sufficient to establish the convergence of the Lorentz force $j^{\gamma}\times B^{\gamma}$ towards $j\times B$, in the sense of distribution. To this end, inspired by \cite{AIM}, we will take a careful analysis of the frequency distribution of magnetic field $B^{\gamma}$, which is based on the spectral properties of Maxwell operator. Unfortunately, it seems hopeless to prove that the operator $$\left(
\begin{array}{cc}
-\frac{{\bf Id}}{\beta\eta^{2}\gamma^{2}}&  \frac{1}{\gamma^{2}}\nabla \times  \\
-\nabla\times &  0 \\
\end{array}\right)$$ is the infinitesimal generator of some contraction semi-group, which means that applying Duhamel's formula may be impossible. Moreover, the modified  antisymmetric operators $$\left(
\begin{array}{cc}
-\frac{{\bf Id}}{\beta\eta^{2}\gamma^{2}}&  \frac{1}{\gamma^{2}}\nabla \times  \\
-\frac{1}{\gamma^{2}}\nabla\times &  0 \\
\end{array}\right) \text{ and } \left(
\begin{array}{cc}
-\frac{{\bf Id}}{\beta\eta^{2}\gamma^{2}}&  \nabla \times  \\
-\nabla\times &  0 \\
\end{array}\right),$$ both of them can be the infinitesimal generator of some contraction semi-group. However, the analysis of the frequencies of $B^{\gamma}$ seems difficult since the tricky term $B^{\gamma}$ in the right-hand side can't be canceled out according to the Duhamel's formula. Luckily, the same problem doesn't trouble us when we choose the modified antisymmetric operator as $$\left(
\begin{array}{cc}
-\frac{{\bf Id}}{\beta\eta^{2}\gamma^{2}}&  \frac{1}{\gamma}\nabla \times  \\
-\frac{1}{\gamma}\nabla\times &  0 \\
\end{array}\right)$$ (see \eqref{Maxwell operator} in Section \ref{sec2}), since the term $B^{\gamma}$ in the right-hand side will vanish after integrating by parts (see \eqref{B1} in Section \ref{sec3}). Furthermore, the systems we are dealing with are more complex and require more sophisticated analysis and energy estimates.

The rest of the paper is arranged as follows. In Section \ref{sec2}, we establish spectral analysis of linear Maxwell system in \eqref{con:NSM} and two useful lemmas. We provide a rigorous proof of Theorem \ref{Mian} in Section \ref{sec3}.

\section{Spectral analysis}\label{sec2}
Now, let's move on to the detailed spectral analysis on the Maxwell operator.
Define
\begin{equation}\label{G}
\begin{cases}
G_{1}:=\left(\frac{1}{\gamma^{2}}-\frac{1}{\gamma}\right)\nabla\times B+\frac{1}{\beta\eta^{2}\gamma^{2}}P\left(\partial _tu+{\rm div}G_{4}-\Delta u-\frac{1}{\eta}G_{3}\right), \\
G_{2}:=\left(\frac{1}{\gamma}-1\right)\nabla\times E=\left(1-\frac{1}{\gamma}\right)\partial _tB, \\
G_{3}:=u\times B, \\
G_{4}:=u\otimes u,
\end{cases}
\end{equation}
where $P$ is the Leray projector onto divergence-free vector fields. Clearly, the Maxwell's system in \eqref{con:NSM} can be rewritten as
\begin{equation}\label{maxwell}
\begin{cases}
\partial _tE=-\frac{1}{\beta\eta^{2}\gamma^{2}}E+\frac{1}{\gamma}\nabla\times B+G_{1}, \\
\partial _tB=\frac{1}{\gamma}\nabla\times E+G_{2},\\
{\rm div}B=0,
\end{cases}
\end{equation}
or, equivalently,
\begin{equation}\label{max}
\partial _t \left(
\begin{array}{c}
E\\
B\\
\end{array}
\right)= {\bf A} \left(
\begin{array}{c}
E\\
B\\
\end{array}\right)
+\left(
\begin{array}{c}
G_{1}\\
G_{2}\\
\end{array}\right),
\end{equation}
where Maxwell's operator ${\bf A}$ is given by
\begin{equation}\label{Maxwell operator}
{\bf A}:=\left(
\begin{array}{cc}
-\frac{{\bf Id}}{\beta\eta^{2}\gamma^{2}}&  \frac{1}{\gamma}\nabla \times  \\
-\frac{1}{\gamma}\nabla\times &  0
\end{array}\right).
\end{equation}
More precisely, define
$$X:=\left\{\left(E,B\right)\in\left(L^{2}\left(\mathbb{R}^{3}\right)\right)^{2}: {\rm div} B = 0\right\},$$
$$\mathcal{D}\left(A\right):=\left\{\left(E,B\right)\in X:\left(PE,B\right)\in \left(H^{1}\left(\mathbb{R}^{3}\right)\right)^{2}\right\}\subset X,$$
then one can verify that the unbounded linear operator ${\bf A}:\mathcal{D}({\bf A})\rightarrow X$ is closed and that $\mathcal{D}({\bf A})$ is dense in $X$. Moreover, for any $\lambda > 0$, a direct computation shows that the operator
$$\lambda {\bf Id}-{\bf A}:\mathcal{D}({\bf A})\rightarrow X$$
is bijective and that $\left\| \left(\lambda {\bf Id}-{\bf A}\right)^{-1}\right\| \leqslant \frac {1}{\lambda}.$ In view of the Hille-Yosida theorem, there exists a contraction semigroup denoted by $\left\{e^{t{\bf A}}\right\}_{t\geqslant 0}$ with ${\bf A}$ being the infinitesimal generator of it. Therefore, for any initial data $(E^{0}, B^{0})\in\mathcal{D}({\bf A})$ and any inhomogeneous term $ \left(G_{1}, G_{2}\right)\in C^{1}\left(\left(\mathbb{R}^{+}\right),X\right),$ the solution to \eqref{max} can be given by Duhamel's formula:
\begin{equation}\label{duhamel}
\left(
\begin{array}{c}
E\\
B\\
\end{array}
\right)= e^{t{\bf A}} \left(
\begin{array}{c}
E^{0}\\
B^{0}\\
\end{array}\right)
+\int_{0}^{t}e^{(t-\tau){\bf A}}\left(
\begin{array}{c}
G_{1}\\
G_{2}\\
\end{array}\right)
(\tau)d\tau.
\end{equation}
Moreover, the Duhamel's formula \eqref{duhamel} remains well-defined provided $\left(G_{1}, G_{2}\right)\in L^{1}\left(0, T, X\right),$ $\left(E^{0}, B^{0}\right)\in X$. In fact, $e^{(t-\tau){\bf A}}\left(
\begin{array}{c}
G_{1}\\
G_{2}\\
\end{array}\right)
(\tau)$ is strongly measurable, since it is weakly measurable for any $t>0$, and $L^{2}\left(\mathbb{R}^{3}\right)$ is separable. Notice that
$$\tau\rightarrow \left\| e^{(t-\tau){\bf A}}\left(
\begin{array}{c}
G_{1}\\
G_{2}\\
\end{array}\right)(\tau)\right\|_{X}$$
is integrable, therefore
$$e^{(t-\tau){\bf A}}\left(
\begin{array}{c}
G_{1}\\
G_{2}\\
\end{array}\right)(\tau)$$ is Bochner integrable.
Furthermore, one can easily verify that \eqref{duhamel} is a weak solution of \eqref{max} through approximation arguments in this setting.

In order to conduct a spectral analysis of ${\bf A}$, taking the Fourier transform  on \eqref{duhamel}, we get
\begin{equation}\label{fmax}
\partial _t \left(
\begin{array}{c}
\hat{E}\\
\hat{B}\\
\end{array}
\right)= \hat{{\bf A}}(\xi) \left(
\begin{array}{c}
\hat{E}\\
\hat{B}\\
\end{array}\right)
+\left(
\begin{array}{c}
\hat{G_{1}}\\
\hat{G_{2}}\\
\end{array}\right),
\end{equation}
where
$$
\hat{{\bf A}}(\xi)=\left(
\begin{array}{cc}
-\frac{{\bf Id}}{\beta\eta^{2}\gamma^{2}}&  \frac{1}{\gamma}i\xi \times  \\
-\frac{1}{\gamma}i\xi\times &  0 \\
\end{array}\right).$$
More precisely, for every $\xi\in \mathbb{R}^{3}\setminus\{0\}$, following \cite{AIM}, we define the 5-dimensional vector subspace of $\mathbb{C}^{3}\times\mathbb{C}^{3}$ as:
$$\mathcal{E}(\xi):=\left\{(e, b)\in\mathbb{C}^{3}\times\mathbb{C}^{3}:\xi\cdot b=0\right\},$$
then $\hat{{\bf A}}(\xi):\mathcal{E}(\xi)\rightarrow \mathcal{E}(\xi)$ is a linear finite-dimensional operator.

Now, we are ready to give a detailed analysis on the properties of the semigroup.
\begin{Lemma}\label{Lemma 2.1} For any $\xi\in \mathbb{R}^{3}\setminus\{0\}$, such that $|\xi|\neq\frac{1}{2\beta\eta^{2}\gamma}$, there are three different eigenvalues of $\hat{{\bf A}}(\xi)$ denoted by $\lambda_{0}=-\frac{1}{\beta\eta^{2}\gamma^{2}}$, $\lambda_{+}(\xi)$, $\lambda_{-}(\xi)$ with
\begin{equation}\label{lamuda}
\lambda_{\pm}(\xi)=\frac{-1\pm\sqrt{1-4\beta^{2}\eta^{4}\gamma^{2}|\xi|^{2}}}{2\beta\eta^{2}\gamma^{2}}.
\end{equation}
Additionally, the basis of the subspace $\mathcal{E}(\xi)$ can be made up by eigenvectors, which means $\hat{{\bf A}}(\xi)$ is diagonalizable, moreover, the eigenspaces corresponding to $\lambda_{0}$, $\lambda_{+}(\xi)$ and $\lambda_{-}(\xi)$ are respectively given by
\begin{align*}
\mathcal{E}_{0}(\xi)={\rm span}\left\{\left(
\begin{array}{c}
\xi\\
0 \\
\end{array}
\right)\right\},
\end{align*}
\begin{align}\label{lamuda+}
\mathcal{E}_{+}(\xi)&=\left\{\left(
\begin{array}{c}
e\\
-\frac{i}{\gamma\lambda_{+}}\xi\times e \\
\end{array}
\right)\in\mathbb{C}^{3}\times\mathbb{C}^{3}: e\in\mathbb{C}^{3},\;\;\xi\cdot e=0\right\}\notag\\
&=\left\{\left(
\begin{array}{c}
-\frac{i}{\gamma\lambda_{-}}\xi\times b\\
 b\\
\end{array}
\right)\in\mathbb{C}^{3}\times\mathbb{C}^{3}: b\in\mathbb{C}^{3},\;\;\xi\cdot b=0\right\},
\end{align}
\begin{align}\label{lamuda-}
\mathcal{E}_{-}(\xi)&=\left\{\left(
\begin{array}{c}
e\\
-\frac{i}{\gamma\lambda_{-}}\xi\times e \\
\end{array}
\right)\in\mathbb{C}^{3}\times\mathbb{C}^{3}: e\in\mathbb{C}^{3},\;\;\xi\cdot e=0\right\}\notag\\
&=\left\{\left(
\begin{array}{c}
-\frac{i}{\gamma\lambda_{+}}\xi\times b\\
 b\\
\end{array}
\right)\in\mathbb{C}^{3}\times\mathbb{C}^{3}:b\in\mathbb{C}^{3},\;\; \xi\cdot b=0\right\}.
\end{align}
For any $\xi\in \mathbb{R}^{3}\setminus\{0\}$ with $|\xi|=\frac{1}{2\beta\eta^{2}\gamma}$, there are two different eigenvalues of $\hat{{\bf A}}(\xi)$ denoted by $\lambda_{0}=-\frac{1}{\beta\eta^{2}\gamma^{2}}$, $\lambda_{1}=-\frac{1}{2\beta\eta^{2}\gamma^{2}}$, and the corresponding eigenspaces are respectively given by
\begin{align*}\mathcal{E}_{0}(\xi)={\rm span}\left\{\left(
\begin{array}{c}
\xi\\
0 \\
\end{array}
\right)\right\},
\end{align*}
\begin{align}\label{lamuda1}\mathcal{E}_{1}(\xi)&=\left\{\left(
\begin{array}{c}
e\\
-\frac{i}{\lambda_{1}\gamma}\xi\times e \\
\end{array}
\right)\in\mathbb{C}^{3}\times\mathbb{C}^{3}: e\in\mathbb{C}^{3},\;\;\xi\cdot e=0\right\}\notag\\
&=\left\{\left(
\begin{array}{c}
-\frac{i}{\lambda_{1}\gamma}\xi\times b\\
 b\\
\end{array}
\right)\in\mathbb{C}^{3}\times\mathbb{C}^{3} : b\in\mathbb{C}^{3},\;\;\xi\cdot b=0\right\}.
\end{align}
In this case, $\hat{{\bf A}}(\xi)$ is not diagonalizable.
\end{Lemma}
\begin{proof}
Suppose
$$\lambda\left(
\begin{array}{c}
e\\
b\\
\end{array}
\right)= \hat{{\bf A}}(\xi) \left(
\begin{array}{c}
e\\
b\\
\end{array}\right),$$
where $\lambda$ and
$\left(
\begin{array}{c}
e\\
b\\
\end{array}
\right)$
are eigenvalue and eigenvector of $\hat{{\bf A}}(\xi)$, respectively. Then we have
\begin{equation}\label{q1}
\begin{cases}
\left(\lambda+\frac{1}{\beta\eta^{2}\gamma^{2}}\right)e=\frac{i}{\gamma}\xi\times b,\\
\lambda b=-\frac{i}{\gamma}\xi\times e.
\end{cases}
\end{equation}
Firstly, $e\neq 0$, otherwise $b=0$. Notice that $\xi\cdot b=0$, then $b=0$ and $\xi\times e=0$ provided $\lambda=-\frac{1}{\beta\eta^{2}\gamma^{2}}$. Thus $\lambda=-\frac{1}{\beta\eta^{2}\gamma^{2}}$ is an eigenvalue of $\hat{{\bf A}}(\xi)$ with the corresponding eigenspace given by
$\mathcal{E}_{0}(\xi)={\rm span}\left\{\left(
\begin{array}{c}
\xi\\
0 \\
\end{array}
\right)\right\}.$

Now, setting $\lambda\neq-\frac{1}{\beta\eta^{2}\gamma^{2}}$, multiplying \eqref{q1} by $\xi$, we have
$$\xi\cdot e=0, \;\; \beta\eta^{2}\gamma^{2}\lambda^{2}+\lambda +\beta\eta^{2}\xi^{2}=0,$$
whose roots $\lambda_{+}(\xi)$ and $\lambda_{-}(\xi)$ are exactly given by \eqref{lamuda}.

While $|\xi|=\frac{1}{2\beta\eta^{2}\gamma}$, $\lambda_{+}(\xi)=\lambda_{-}(\xi)=-\frac{1}{2\beta\eta^{2}\gamma^{2}}$ and the corresponding eigenspace $\mathcal{E}_{1}(\xi)$ is defined by \eqref{lamuda1}, then one can verify that the operator $\hat{{\bf A}}(\xi)$ is not diagonalizable in this condition.

Next, setting $|\xi|\neq\frac{1}{2\beta\eta^{2}\gamma}$, $\lambda_{+}(\xi)$ and $\lambda_{-}(\xi)$ are different. The corresponding  eigenspaces $\mathcal{E}_{+}(\xi)$ and $\mathcal{E}_{-}(\xi)$ are given by \eqref{lamuda+} and \eqref{lamuda-}, respectively. Therefore, $\hat{{\bf A}}(\xi)$ is diagonalizable in this case. Furthermore, these eigenspaces are not orthogonal to each other since the operator $\hat{{\bf A}}(\xi)$ is not symmetric. We complete the proof of Lemma \ref{Lemma 2.1}.
\end{proof}
\begin{Lemma}\label{lemma3.2}
Suppose that $1< K <2$ is a fixed constant. For any $\xi\in\mathbb{R}^{3}\setminus\{0\}$, $\lambda_{\pm}(\xi)$ are the eigenvalues of $\hat{\bf A}(\xi)$ defined by \eqref{lamuda}, then we have the following estimates.

If $|\xi|\leqslant \frac{1}{2\beta\eta^{2}\gamma}$,
$$-2\beta\eta^{2}|\xi|^{2}\leqslant\lambda_{+}(\xi)\leqslant-\beta\eta^{2}|\xi|^{2},\;\;
-\frac{1}{\beta\eta^{2}\gamma^{2}}\leqslant\lambda_{-}(\xi)\leqslant
-\frac{1}{2\beta\eta^{2}\gamma^{2}},$$
moreover, if $|\xi|\leqslant \frac{1}{2K\beta\eta^{2}\gamma}$,
$$\left|\frac{\lambda_{-}(\xi)}{\lambda_{-}(\xi)-\lambda_{+}(\xi)}\right|\leqslant\frac{K}{\sqrt{K^{2}-1}},\;\;
\left|\frac{\lambda_{+}(\xi)}{\lambda_{-}(\xi)-\lambda_{+}(\xi)}\right|\lesssim \gamma^{2}|\xi|^{2}.$$

If $|\xi|> \frac{1}{2\beta\eta^{2}\gamma}$,
$$\left|\lambda_{+}(\xi)\right|=\left|\lambda_{-}(\xi)\right|=\frac{\left|\xi\right|}{\gamma},
\;\;\Re\left(\lambda_{+}(\xi)\right)=\Re\left(\lambda_{-}(\xi)\right)=-\frac{1}{2\beta\eta^{2}\gamma^{2}},$$
moreover, if  $|\xi|\geqslant \frac{K}{2\beta\eta^{2}\gamma}$,
$$\left|\frac{\lambda_{\pm}(\xi)}{\lambda_{-}(\xi)-\lambda_{+}(\xi)}\right|\leqslant\frac{K}{2\sqrt{K^{2}-1}},\;\;
\left|\frac{1}{\lambda_{-}(\xi)-\lambda_{+}(\xi)}\right|\leqslant \frac{K}{2\sqrt{K^{2}-1}}\frac{\gamma}{\left|\xi\right|}.$$
\end{Lemma}
\begin{proof} We first consider the case $|\xi|\leqslant \frac{1}{2\beta\eta^{2}\gamma}$. Note that
\begin{align*}
\lambda_{+}(\xi)&=\frac{-1+\sqrt{1-4\beta^{2}\eta^{4}\gamma^{2}|\xi|^{2}}}{2\beta\eta^{2}\gamma^{2}}
=\frac{-4\beta^{2}\eta^{4}\gamma^{2}|\xi|^{2}}{2\beta\eta^{2}\gamma^{2}\left(1+\sqrt{1-4\beta^{2}\eta^{4}\gamma^{2}|\xi|^{2}}\right)}\\
&=\frac{-2\beta\eta^{2}|\xi|^{2}}{1+\sqrt{1-4\beta^{2}\eta^{4}\gamma^{2}|\xi|^{2}}},
\end{align*}
then we have $-2\beta\eta^{2}|\xi|^{2}\leqslant\lambda_{+}(\xi)\leqslant-\beta\eta^{2}|\xi|^{2}$. The similar estimate on $\lambda_{-}(\xi)$ is trivial.
Furthermore, one has
\begin{align*}
\left|\frac{\lambda_{-}(\xi)}{\lambda_{-}(\xi)-\lambda_{+}(\xi)}\right|
=\frac{1+\sqrt{1-4\beta^{2}\eta^{4}\gamma^{2}|\xi|^{2}}}{2\sqrt{1-4\beta^{2}\eta^{4}\gamma^{2}|\xi|^{2}}}
\leqslant \frac{1}{\sqrt{1-4\beta^{2}\eta^{4}\gamma^{2}|\xi|^{2}}}
\leqslant\frac{K}{\sqrt{K^{2}-1}},
\end{align*}
\begin{align*}
\left|\frac{\lambda_{+}(\xi)}{\lambda_{-}(\xi)-\lambda_{+}(\xi)}\right|
&=\frac{1-\sqrt{1-4\beta^{2}\eta^{4}\gamma^{2}|\xi|^{2}}}{2\sqrt{1-4\beta^{2}\eta^{4}\gamma^{2}|\xi|^{2}}}\\
&=\frac{2\beta^{2}\eta^{4}\gamma^{2}|\xi|^{2}}{\sqrt{1-4\beta^{2}\eta^{4}\gamma^{2}|\xi|^{2}}\left(1+\sqrt{1-4\beta^{2}\eta^{4}\gamma^{2}|\xi|^{2}}\right)}\\
&\lesssim \gamma^{2}|\xi|^{2},
\end{align*}
provided $|\xi|\leqslant \frac{1}{2K\beta\eta^{2}\gamma}$.

Next, focusing on the case $|\xi|> \frac{1}{2\beta\eta^{2}\gamma}$, writing
$$\lambda_{\pm}(\xi)=\frac{1}{2\beta\eta^{2}\gamma^{2}}\left(-1\pm i\sqrt{4\beta^{2}\eta^{4}\gamma^{2}|\xi|^{2}-1}\right),$$
we get
$$\left|\lambda_{+}(\xi)\right|=\left|\lambda_{-}(\xi)\right|=\frac{|\xi|}{\gamma},\;\;\Re\left(\lambda_{+}(\xi)\right)
=\Re\left(\lambda_{-}(\xi)\right)=-\frac{1}{2\beta\eta^{2}\gamma^{2}}.$$
Furthermore, assuming $|\xi|\geqslant \frac{K}{2\beta\eta^{2}\gamma}$, one has
$$\left|\frac{\lambda_{\pm}(\xi)}{\lambda_{-}(\xi)-\lambda_{+}(\xi)}\right|
=\frac{\beta\eta^{2}\gamma^{2}}{\sqrt{4\beta^{2}\eta^{4}\gamma^{2}|\xi|^{2}-1}}\frac{|\xi|}{\gamma}
=\frac{1}{\sqrt{4-\frac{1}{\beta^{2}\eta^{4}\gamma^{2}|\xi|^{2}}}}
\leqslant\frac{K}{2\sqrt{{K^{2}-1}}},$$
and the last estimate is followed.
\end{proof}
\section{Proof of the main result}\label{sec3}
In this section, we study the asymptotic limit of the incompressible Naiver-Stokes-Maxwell system with generalized Ohm's law \eqref{con:NSM} based on the spectral analysis established in Section \ref{sec2}.

Here, we consider a family of global and finite energy weak solutions $(u^{\gamma},E^{\gamma},B^{\gamma})$  to \eqref{con:NSM} with uniformly bounded initial data:
$$\left(u^{0\gamma},E^{0\gamma},B^{0\gamma}\right)\in L^{2}(\mathbb{R}^{3})\times L^{2}(\mathbb{R}^{3})\times L^{2}(\mathbb{R}^{3}).$$
Recall that the formal energy conservation law \eqref{energy} implies that the following uniform bounds of the weak solutions $(u^{\gamma},E^{\gamma},B^{\gamma})$:
$$u^{\gamma}\in L^{\infty}_{t}L^{2}_{x}\cap L^{2}_{t}\dot{H}^{1}_{x},\; \left(B^{\gamma},\gamma E^{\gamma}\right)\in L^{\infty}_{t}L^{2}_{x},\;j^{\gamma}\in L^{2}_{t,x}.$$
Therefore, up to extraction of subsequences, we have the following weak convergences as $\gamma\rightarrow0:$
\begin{align*}
\left(u^{0\gamma},E^{0\gamma},B^{0\gamma}\right)& \rightharpoonup \left(u^{0},E^{0},B^{0}\right) \text{ in } L^{2}_{x},\\
\left(u^{\gamma}, B^{\gamma}\right)& \rightharpoonup^{\ast} (u,B) \text{ in } L^{\infty}_{t}L^{2}_{x},\\
j^{\gamma}& \rightharpoonup j \text{ in } L^{2}_{t,x}.
\end{align*}
From the Faraday's equation in \eqref{con:NSM}, we know that
$$E^{\gamma}=-\partial_{t}A^{\gamma},\text{ where } \;\nabla\times A^{\gamma}=B^{\gamma}\text{ and } {\rm div}A^{\gamma}=0,$$
which implies $E^{\gamma}$ enjoys some kind of uniform weak bound. Thus, the weak stability of linear terms in \eqref{con:NSM} is deduced. Furthermore, combining the generalized Ohm's law and the momentum conservation equation in \eqref{con:NSM}, we find that
\begin{equation}\label{partialtu}
\partial_{t}u^{\gamma}\;{\rm is\; uniformly\; bounded\;in}\;L^{2}_{t,\rm loc}H^{-3}_{x}.
\end{equation}
By virtue of the Sobolev embedding and the regularity hypothesis \eqref{regularhp}, $u^{\gamma}$ is uniformly bounded in
$$L_{t}^{\infty}L_{x}^{2}\cap L_{t}^{2}\dot{H}_{x}^{1}\cap L_{t}^{2}\dot{H}_{x}^{1+s}\subset L_{t}^{\infty}L_{x}^{2}\cap L_{t}^{2}L_{x}^{6}\cap L_{t}^{2}L^{\infty}_{x},$$
then a routine application of classical compactness result by Aubin and Lions \cite{SJ} implies that $u$ is relatively compact in the strong topology of $L^{2}_{t,x,\rm loc}$. Inserting the bounds on $B^{\gamma}$, we have
\begin{align}\label{bddg3g4}
u^{\gamma}\times B^{\gamma}\in L_{t}^{2}L_{x}^{2},\;\;u^{\gamma}\otimes u^{\gamma}\in L_{t}^{2}L_{x}^{2},\;\;\nabla \left(u^{\gamma}\otimes u^{\gamma}\right)\in L_{t}^{2}L_{x}^{\frac{3}{3-s}},
\end{align}
whence,
\begin{align}\label{wconvg3g4}
G^{\gamma}_{3}:=u^{\gamma}\times B^{\gamma}\rightharpoonup G_{3}:=u\times B\;\;{\rm in} \;\;L^{2}_{t,x},\notag\\
G^{\gamma}_{4}:=u^{\gamma}\otimes u^{\gamma}\rightharpoonup G_{4}:=u\otimes u\;\;{\rm in} \;\;L^{2}_{t,x}.
\end{align}

From the Faraday's law, one has
$$\partial_{t}B^{\gamma}=\nabla\times G_{3}^{\gamma}-\beta\eta^{2}\nabla\times j^{\gamma}-\eta\nabla\times\left(j^{\gamma}\times B^{\gamma}\right) ,$$
 then $B^{\gamma}$ is relatively compact in $C([0,T], H^{-1})$ through standard compactness arguments.
Therefore, taking the weak limit in the Faraday's law, one obtains that $B\in L^{\infty}_{t}L^{2}_{x}$ solves
\begin{equation*}
\partial_{t}B=\beta\eta^{2}\Delta B+\nabla\times G_{3}-\eta\nabla\times\left(\partial_{t}u+{\rm div}G_{4}-\Delta u\right)
\end{equation*}
with the initial data $B^{0}\in L^{2}_{x}$.
Finally, we have from Duhamel's principle that:
$$B=e^{\beta\eta^{2}\Delta t}B^{0}+\int_{0}^{t}e^{\beta\eta^{2}\Delta (t-\tau)}\left[\nabla\times G_{3}-\eta\nabla\times\left(\partial_{\tau}u+{\rm div}G_{4}-\Delta u\right)\right]d\tau,$$
then, after taking Fourier transformation, we have
\begin{equation}\label{B}
\hat{B}=e^{-\beta\eta^{2}|\xi|^{2} t}\hat{B}^{0}+\int_{0}^{t}e^{-\beta\eta^{2}|\xi|^{2} (t-\tau)}\left[i\xi\times \hat{G}_{3}-\eta i\xi\times\left(\partial_{\tau}\hat{u}+i\xi\cdot \hat{G}_{4}+|\xi|^{2} u\right)\right]d\tau.
\end{equation}

Thus, in order to establish the asymptotic limit of the Navier-Stokes-Maxwell system \eqref{con:NSM}, there only remains to establish the weak stability of the Lorentz force $j^{\gamma}\times B^{\gamma}$, which will follow from a careful combination of frequency analysis on $B^{\gamma}$ and the hypothesis \eqref{higj} on the very high frequencies of $j$. Therefore, we decompose the initial data $(\hat{E}^{0\gamma},\hat{B}^{0\gamma})\in X$ and the inhomogeneous terms $(\hat{G}_{1}^{\gamma},\hat{G}_{2}^{\gamma})\in L^{1}(0,T,X)$ based on the spectral analysis in Section \ref{sec2}.

Applying Lemma \ref{Lemma 2.1}, for almost every $\xi\in \mathbb{R}^{3}$, one has
\begin{equation}\label{EB}
\left(
\begin{array}{c}
\hat{E}^{0\gamma}\\
\hat{B}^{0\gamma}\\
\end{array}
\right)=\left(
\begin{array}{c}
\frac{\hat{E}^{0\gamma}\cdot\xi}{|\xi|^{2}}\xi\\
0\\
\end{array}
\right)+\left(
\begin{array}{c}
-\frac{i}{\gamma\lambda_{-}}\xi\times b^{0\gamma}\\
 b^{0\gamma}\\
\end{array}
\right)+\left(
\begin{array}{c}
e^{0\gamma}\\
-\frac{i}{\gamma\lambda_{-}}\xi\times e^{0\gamma} \\
\end{array}
\right),
\end{equation}
or, equivalently,
\begin{equation}\label{eb}
\begin{cases}
se^{0\gamma}=\hat{E}^{0\gamma}+\frac{i}{\gamma \lambda_{-}}\xi\times \hat{B}^{0\gamma}, \\
sb^{0\gamma}=\hat{B}^{0\gamma}+\frac{i}{\gamma \lambda_{-}}\xi\times \hat{E}^{0\gamma},\\
\end{cases}
\end{equation}
where $s=\frac{\lambda_{-}(\xi)-\lambda_{+}(\xi)}{\lambda_{-}(\xi)}$, $\xi\cdot e^{0\gamma}=\xi\cdot b^{0\gamma}=0$. Therefore, according to Lemma \ref{lemma3.2}, we get
\begin{equation}\label{esteb}
\left|se^{0\gamma}\right|\leqslant \left|\hat{E}^{0\gamma}\right|+\left|\hat{B}^{0\gamma}\right|,\;\;
\left|sb^{0\gamma}\right|\leqslant\left|\hat{E}^{0\gamma}\right|+\left|\hat{B}^{0\gamma}\right|.
\end{equation}
Similarly, define
\begin{equation}\label{FGs}
\begin{cases}
\hat{G}_{1}^{\gamma}:=\frac{1}{\gamma^{2}}\left\{\left(1-\gamma\right)i\xi\times \hat{B}^{\gamma}+\frac{1}{\beta\eta}\left[\partial_{t}\hat{u}^{\gamma}+i\xi\cdot \hat{G}_{4}^{\gamma}+|\xi|^{2}\hat{u}^{\gamma}-\frac{1}{\eta}\hat{G}_{3}^{\gamma}\right]\right\},\\
\hat{G}_{2}^{\gamma}:=\left(1-\frac{1}{\gamma}\right)\partial_{t}\hat{B}^{\gamma},\\
\end{cases}
\end{equation}
and then applying Lemma \ref{Lemma 2.1}, one has
\begin{equation}\label{G1G2}
\left(
\begin{array}{c}
\hat{G}_{1}^{\gamma}\\
\hat{G}_{2}^{\gamma}\\
\end{array}
\right)=\left(
\begin{array}{c}
\frac{\hat{G}_{1}^{\gamma}\cdot\xi}{|\xi|^{2}}\xi\\
0\\
\end{array}
\right)+\left(
\begin{array}{c}
-\frac{i}{\gamma\lambda_{-}}\xi\times b^{\gamma}\\
 b^{\gamma}\\
\end{array}
\right)+\left(
\begin{array}{c}
e^{\gamma}\\
-\frac{i}{\gamma\lambda_{-}}\xi\times e^{\gamma} \\
\end{array}
\right),
\end{equation}
for almost every $\xi\in \mathbb{R}^{3}$, or, equivalently,
\begin{equation}\label{G1aG2}
\begin{cases}
se^{\gamma}=\hat{G}_{1}^{\gamma}+\frac{i}{\gamma \lambda_{-}}\xi\times \hat{G}_{2}^{\gamma}, \\
sb^{\gamma}=\hat{G}_{2}^{\gamma}+\frac{i}{\gamma \lambda_{-}}\xi\times \hat{G}_{1}^{\gamma},\\
\end{cases}
\end{equation}
where $\xi\cdot e^{\gamma}=\xi\cdot b^{\gamma}=0$.
 Hence, from Lemma \ref{lemma3.2}, we obtain
\begin{equation}\label{es}
\left|se^{\gamma}\right|\leqslant \left|\hat{G}_{1}^{\gamma}\right|+\left|\hat{G}_{2}^{\gamma}\right|,\;\;
\left|sb^{\gamma}\right|\leqslant \left|\hat{G}_{1}^{\gamma}\right|+\left|\hat{G}_{2}^{\gamma}\right|.
\end{equation}
Letting the operator $e^{t\hat{\bf A}}$ acts on \eqref{EB} and on \eqref{G1G2} respectively, combining Lemma \ref{Lemma 2.1}, then we have
\begin{equation*}
e^{t\hat{\bf A}}\left(
\begin{array}{c}
\hat{E}^{0\gamma}\\
\hat{B}^{0\gamma}\\
\end{array}
\right)=e^{-\frac{1}{\beta\eta^{2}\gamma^{2}}}\left(
\begin{array}{c}
\frac{\hat{E}^{0}\cdot\xi}{|\xi|^{2}}\xi\\
0\\
\end{array}
\right)+e^{t\lambda_{+}}\left(
\begin{array}{c}
-\frac{i}{\gamma\lambda_{-}}\xi\times b^{0\gamma}\\
 b^{0\gamma}\\
\end{array}
\right)+e^{t\lambda_{-}}\left(
\begin{array}{c}
e^{0\gamma}\\
-\frac{i}{\gamma\lambda_{-}}\xi\times e^{0\gamma}
\end{array}
\right),
\end{equation*}
and
\begin{equation*}
e^{t\hat{\bf A}}\left(
\begin{array}{c}
\hat{G}_{1}^{\gamma}\\
\hat{G}_{2}^{\gamma}\\
\end{array}
\right)=e^{-\frac{1}{\beta\eta^{2}\gamma^{2}}}\left(
\begin{array}{c}
\frac{\hat{G}_{1}^{\gamma}\cdot\xi}{|\xi|^{2}}\xi\\
0\\
\end{array}
\right)+e^{t\lambda_{+}}\left(
\begin{array}{c}
-\frac{i}{\gamma\lambda_{-}}\xi\times b^{\gamma}\\
 b^{\gamma}\\
\end{array}
\right)+e^{t\lambda_{-}}\left(
\begin{array}{c}
e^{\gamma}\\
-\frac{i}{\gamma\lambda_{-}}\xi\times e^{\gamma}
\end{array}
\right) .
\end{equation*}
Substituting them into \eqref{duhamel} implies that
\begin{equation*}
\hat{B}^{\gamma}=e^{t\lambda_{+}}b^{0\gamma}-\frac{i}{\gamma\lambda_{-}}e^{t\lambda_{-}}\xi\times e^{0\gamma}+\int_{0}^{t}e^{(t-\tau)\lambda_{+}}b^{\gamma}-\frac{i}{\gamma\lambda_{-}}e^{(t-\tau)\lambda_{-}}\xi\times e^{\gamma}d\tau.
\end{equation*}
Notice that
\begin{align*}-\frac{i}{\gamma\lambda_{-}}\xi\times e^{0\gamma}=\hat{B}^{0\gamma}-b^{0\gamma},\\
-\frac{i}{\gamma\lambda_{-}}\xi\times e^{\gamma}=\hat{G}_{2}^{\gamma}-b^{\gamma},\\
sb^{\gamma}=\hat{G}_{2}^{\gamma}+\frac{i}{\gamma \lambda_{-}}\xi\times \hat{G}_{1}^{\gamma},
\end{align*}
one has
\begin{align*}
\hat{B}^{\gamma}
&=e^{t\lambda_{-}}\hat{B}^{0\gamma}+\left(e^{t\lambda_{+}}-e^{t\lambda_{-}}\right)b^{0\gamma} +\int_{0}^{t}\left(e^{(t-\tau)\lambda_{+}}-e^{(t-\tau)\lambda_{-}}\right)b^{\gamma}+e^{(t-\tau)\lambda_{-}}\hat{G}_{2}^{\gamma}d\tau\\
&=e^{t\lambda_{-}}\hat{B}^{0\gamma}+\left(e^{t\lambda_{+}}-e^{t\lambda_{-}}\right)b^{0\gamma} +\int_{0}^{t}\left[\frac {\lambda_{-}}{\lambda_{-}-\lambda_{+}} e^{(t-\tau) \lambda_{+}}-\frac{\lambda_{+}}{\lambda_{-}-\lambda_{+}}e^{(t-\tau)\lambda_{-}}\right] \hat{G}_{2}^{\gamma}d \tau\\
&\quad+\frac{i}{\gamma (\lambda_{-}-\lambda_{+})}\int_{0}^{t}\left(e^{(t-\tau)\lambda_{+}}-e^{(t-\tau)\lambda_{-}}\right)\xi \times \hat{G}_{1}^{\gamma} d \tau .
\end{align*}
Combining \eqref{FGs} and integrating by parts in $t$ yields that
\begin{align}\label{B1}
\hat{B}^{\gamma}
&=\gamma e^{t\lambda_{-}}\hat{B}^{0\gamma}+\gamma(e^{t\lambda_{+}}-e^{t\lambda_{-}})b^{0\gamma}
+\frac{1-\gamma}{\lambda_{-}-\lambda_{+}}\left(\lambda_{-}e^{t\lambda_{+}}-\lambda_{+}e^{t\lambda_{-}}\right)\hat{B}^{0\gamma}\notag\\
&\quad +\frac{i}{\beta\eta\gamma^{2}(\lambda_{-}\!-\!\lambda_{+})}\int_{0}^{t}\left(e^{(t-\tau)\lambda_{+}}\!-\!e^{(t-\tau)\lambda_{-}}\right)\xi\times \left(\partial_{\tau}\hat{u}^{\gamma}+i\xi\cdot \hat{G}_{4}^{\gamma}+|\xi|^{2}\hat{u}^{\gamma}-\frac{1}{\eta}\hat{G}_{3}^{\gamma}\right)d\tau,
\end{align}
where the term $\hat{B}^{\gamma}$ in the right-hand side vanishes.
In order to cancel out the term $\partial_{t}\hat{u}^{\gamma}$, we integrate by parts in $t$ again to get
\begin{align}\label{B2}
\hat{B}^{\gamma}
&=\gamma e^{t\lambda_{-}}\hat{B}^{0\gamma}+\frac{1-\gamma}{\lambda_{-}-\lambda_{+}}\left(\lambda_{-}e^{t\lambda_{+}}-\lambda_{+}e^{t\lambda_{-}}\right)\hat{B}^{0\gamma}\notag\\
&\quad+\gamma\left(e^{t\lambda_{+}}-e^{t\lambda_{-}}\right)b^{0\gamma}
-\frac{i}{\beta\eta\gamma^{2}\left(\lambda_{-}\!-\!\lambda_{+}\right)}\left(e^{t\lambda_{+}}\!-\!e^{t\lambda_{-}}\right)\xi\times\hat{u}^{0\gamma}\notag\\
&\quad -\frac{i}{\beta\eta\gamma^{2}(\lambda_{-}\!-\!\lambda_{+})}\int_{0}^{t}(\lambda_{-}e^{(t-\tau)\lambda_{-}}\!
-\!\lambda_{+}e^{(t-\tau)\lambda_{+}})\xi\times\hat{u}^{\gamma}d\tau\notag\\
&\quad +\frac{i}{\beta\eta\gamma^{2}(\lambda_{-}-\lambda_{+})}\int_{0}^{t}\left(e^{(t-\tau)\lambda_{+}}-e^{(t-\tau)\lambda_{-}}\right)\xi\times \left(i\xi\cdot \hat{G}_{4}^{\gamma}+|\xi|^{2}\hat{u}^{\gamma}-\frac{1}{\eta}\hat{G}_{3}^{\gamma}\right)d\tau.
\end{align}
Recall that we want to establish the convergence of Lorentz force $j^{\gamma}\times B^{\gamma}$ towards $j\times B$, in the sense of distribution. To this end, we strengthen the convergence of $B^{\gamma}$ towards $B$, which can be given by \eqref{B}. Whence, motivated by \cite{AIM}, we decompose $B^{\gamma}$ into five parts
$$B^{\gamma}=B^{\gamma}_{\ll}+B^{\gamma}_{<}+B^{\gamma}_{\sim}+B^{\gamma}_{>}+B^{\gamma}_{\gg},$$
where
\begin{align*}
&B^{\gamma}_{\ll}=\mathcal{F}^{-1}\left(\chi_{\left\{0\leqslant|\xi|\leqslant R\right\}}\hat{B}^{\gamma}\right),\\
&B^{\gamma}_{<}=\mathcal{F}^{-1}\left(\chi_{\left\{R<|\xi|\leqslant \frac{1}{2K\beta\eta^{2}\gamma}\right\}}\hat{B}^{\gamma}\right),\\
&B^{\gamma}_{\sim}=\mathcal{F}^{-1}\left(\chi_{\left\{\frac{1}{2K\beta\eta^{2}\gamma}<|\xi|\leqslant \frac{K}{2\beta\eta^{2}\gamma}\right\}}\hat{B}^{\gamma}\right),\\
&B^{\gamma}_{>}=\mathcal{F}^{-1}\left(\chi_{\left\{ \frac{K}{2\beta\eta^{2}\gamma}<|\xi|\leqslant \phi\left(\frac{\gamma}{\delta}\right)\right\}}\hat{B}^{\gamma}\right),\\
&B^{\gamma}_{\gg}=\mathcal{F}^{-1}\left(\chi_{\left\{|\xi|> \phi\left(\frac{\gamma}{\delta}\right)\right\}}\hat{B}^{\gamma}\right),
\end{align*}
for some fixed constant $1<K<2$, which will be determined later, for any large radius $0<R<\frac{1}{2K\beta\eta^{2}\gamma}$ and for some small parameter $\delta>0$.
Hence, for any $\varphi\in C_{c}^{\infty}\left(\mathbb{R}^{+}\times \mathbb{R}^{3}\right)$, we have
\begin{align}\label{decompose}
&\int_{\mathbb{R}^{+}\times \mathbb{R}^{3}}\left(j^{\gamma}\times B^{\gamma}-j\times B\right)\cdot\varphi\;dtdx\notag\\
&=\int_{\mathbb{R}^{+}\times \mathbb{R}^{3}}j^{\gamma}\times\left( B^{\gamma}_{\ll}-B_{\ll}\right)\cdot\varphi\;dtdx+\int_{\mathbb{R}^{+}\times \mathbb{R}^{3}}j^{\gamma}\times\left( B_{\ll}-B\right)\cdot\varphi\;dtdx\notag\\
&\quad+\int_{\mathbb{R}^{+}\times \mathbb{R}^{3}}\left(j^{\gamma}-j\right)\times B\cdot\varphi\;dtdx
+\int_{\mathbb{R}^{+}\times \mathbb{R}^{3}}j^{\gamma}\times\left( B^{\gamma}_{<}+B^{\gamma}_{\sim}+B^{\gamma}_{>}\right)\cdot\varphi\;dtdx\notag\\
&\quad+\int_{\mathbb{R}^{+}\times \mathbb{R}^{3}}\left(j^{\gamma}\times B^{\gamma}_{\gg}\right)\cdot\varphi\;dtdx.
\end{align}

Firstly, we handle with the last term in \eqref{decompose}, which can be recast as
\begin{align*}
\int_{\mathbb{R}^{+}\times \mathbb{R}^{3}}\left(j^{\gamma}\times B^{\gamma}_{\gg}\right)\cdot\varphi\;dtdx
&=\int_{\mathbb{R}^{+}\times \mathbb{R}^{3}}j^{\gamma}\times \mathcal{F}^{-1}\left(\chi_{\left\{|\xi|> \phi\left(\frac{\gamma}{\delta}\right)\right\}}\hat{B}^{\gamma}\right)\cdot\varphi\;dtdx \\
&=\int_{\mathbb{R}^{+}\times \mathbb{R}^{3}} j^{\gamma}\times \mathcal{F}\left(\chi_{\left\{|\xi|> \phi\left(\frac{\gamma}{\delta}\right)\right\}}\check{B}^{\gamma}\right)\cdot\varphi \;dtdx \\
&=\int_{\mathbb{R}^{+}\times \mathbb{R}^{3}}\chi_{\left\{|\xi|> \phi\left(\frac{\gamma}{\delta}\right)\right\}}
\mathcal{F}\left(j^{\gamma}\varphi\right)\times\check{B}^{\gamma}\;dtdx\\
&=\int_{\mathbb{R}^{+}\times \mathbb{R}^{3}}\left(j^{\gamma}\varphi\right)_{\gg}\times B^{\gamma}\;dtdx,
\end{align*}
where
$$\left(j^{\gamma}\varphi\right)_{\gg}=\mathcal{F}^{-1}\left(\chi_{\left\{|\xi|> \phi\left(\frac{\gamma}{\delta}\right)\right\}}\mathcal{F}\left(j^{\gamma}\varphi\right)\right).$$
Noted that for any $\delta>0$ and $\varphi\in C_{c}^{\infty}\left(\mathbb{R}^{+}\times \mathbb{R}^{3}\right)$
\begin{align*}
\left\|\left(j^{\gamma}\varphi\right)_{\gg}\right\|^{2}_{L_{t,x}^{2}}
&=\int_{\mathbb{R}^{+}\times \mathbb{R}^{3}}\left(\left(j^{\gamma}\varphi\right)_{\gg}\right)_{\gg}j^{\gamma}\varphi\;dtdx\\
&=\int_{\mathbb{R}^{+}\times \mathbb{R}^{3}}\left(j^{\gamma}\varphi\right)_{\gg}j^{\gamma}\varphi\;dtdx\\
&\lesssim \left\|\left(j^{\gamma}\varphi\right)_{\gg}\right\|_{L_{t,x,\rm loc}^{2}},
\end{align*}
and
\begin{align*}
\left\|\left(j^{\gamma}\varphi\right)_{\gg}\right\|_{L_{t,x,\rm loc}^{2}}
&=\left\|\mathcal{F}^{-1}\left(\int_{\mathbb{R}^{3}}j^{\gamma}(y)\varphi(y)e^{-iy\cdot \xi}dy\chi_{\left\{|\xi|> \phi\left(\frac{\gamma}{\delta}\right)\right\}}\right)\right\|_{L_{t,x,\rm loc}^{2}}\\
&=\left\|\mathcal{F}^{-1}\left(\int_{\mathbb{R}^{3}}j^{\gamma}\left(\frac{z}{|\xi|}\right)\varphi\left(\frac{z}{|\xi|}\right)e^{\frac{-iz\cdot \xi}{|\xi|}}\frac{1}{|\xi|}dz\chi_{\left\{|\xi|> \phi\left(\frac{\gamma}{\delta}\right)\right\}}\right)\right\|_{L_{t,x,\rm loc}^{2}}.
\end{align*}
Letting $\gamma\rightarrow 0$ in the above equality, and applying the assumption \eqref{higj} on the very high frequencies of $j^{\gamma}$,  we have
\begin{equation}\label{convhh}
\limsup\limits_{\gamma\rightarrow 0}\left\|\left(j^{\gamma}\varphi\right)_{\gg}\right\|^{2}_{L_{t,x}^{2}}
\lesssim\limsup\limits_{\gamma\rightarrow 0}\left\|\left(j^{\gamma}\right)_{\gg}\right\|_{L_{t,x,\rm loc}^{2}}=0.
\end{equation}

Now, we deal with the term $B^{\gamma}_{\ll}-B_{\ll}$. For any small $h>0$, applying Lemma \ref{lemma3.2} and the estimates \eqref{esteb} and \eqref{es}, we find from \eqref{B1} that
\begin{align*}
&\left|\hat{B}^{\gamma}_{\ll}(t+h)-\hat{B}^{\gamma}_{\ll}(t)\right|\\
&\lesssim e^{-\frac{1}{2\beta\eta^{2}\gamma^{2}}t}\left(\left|\gamma\hat{B}^{0\gamma}\right|+\left|\gamma b^{0\gamma}\right|+\gamma^{2}|\xi|^{2}\left|\hat{B}^{0\gamma}\right|\right)\chi_{\{0\leqslant|\xi|\leqslant R\}}\\
&\quad+e^{-\beta\eta^{2}|\xi|^{2}t}h|\xi|^{2}\left(\left|\gamma b^{0\gamma}\right|+\left|\hat{B}^{0\gamma}\right|\right)\chi_{\{0\leqslant|\xi|\leqslant R\}}\\
&\quad+\int_{0}^{t}\left[e^{-\beta\eta^{2}|\xi|^{2}(t-\tau)}h|\xi|^{2}+e^{-\frac{1}{2\beta\eta^{2}\gamma^{2}}(t-\tau)}\right]\\
&\qquad
\left(|\xi|\left|\partial_{\tau}\hat{u}^{\gamma}\right|+\left|\xi\right|^{2}\left|\hat{G}_{4}^{\gamma}\right|
+|\xi|^{3}\left|\hat{u}^{\gamma}\right|+|\xi|\left|\hat{G}_{3}^{\gamma}\right|\right)\chi_{\{0\leqslant|\xi|\leqslant R\}}d\tau\\
&\quad+\int_{0}^{\infty}\!\!\chi_{[0,h]}(v)\left(e^{-\beta\eta^{2}|\xi|^{2}v}\!+\!e^{-\frac{1}{2\beta\eta^{2}\gamma^{2}}v}\right)\\
&\qquad
\left(|\xi|\left|\partial_{\tau}\hat{u}^{\gamma}\right|
+\!|\xi|^{2}\left|\hat{G}_{4}^{\gamma}\right|\!
+\!|\xi|^{3}\left|\hat{u}^{\gamma}\right|\!+\!|\xi|\left|\hat{G}_{3}^{\gamma}\right|\right)(t+h-v)\chi_{\{0\leqslant|\xi|\leqslant R\}}dv.
\end{align*}
Thus{\small
\begin{align*}
&\left\|\hat{B}^{\gamma}_{\ll}(t+h)-\hat{B}^{\gamma}_{\ll}(t)\right\|_{L_{t,\rm loc}^{2}L_{\xi}^{2}}\\
&\lesssim \gamma\left\|e^{-\frac{1}{2\beta\eta^{2}\gamma^{2}}t}\right\|_{L^{2}_{t}}
\left(\left\|\hat{E}^{0\gamma}\right\|_{L_{x}^{2}}+\left\|\hat{B}^{0\gamma}\right\|_{L_{x}^{2}}\right)+
h\left\|\left\|e^{-\beta\eta^{2}|\xi|^{2}t}\right\|_{L^{2}_{t}}|\xi|^{2}
\left(\left|\hat{E}^{0\gamma}\right|+\left|\hat{B}^{0\gamma}\right|\right)\chi_{\{0\leqslant|\xi|\leqslant R\}}\right\|_{L^{2}_{\xi}}\\
&\quad+h\left\|\left\|e^{-\beta\eta^{2}|\xi|^{2}t}\right\|_{L^{1}_{t}}\left\|
|\xi|^{3}\left|\partial_{t}\hat{u}^{\gamma}\right|+|\xi|^{4}\left|\hat{G}_{4}^{\gamma}\right|
\!+\!|\xi|^{5}\left|\hat{u}^{\gamma}\right|+|\xi|^{3}\left|\hat{G}_{3}^{\gamma}\right|\right\|_{L^{2}_{t,\rm loc}}\chi_{\{0\leqslant|\xi|\leqslant R\}}\right\|_{L^{2}_{\xi}}\\
&\quad+\left\|e^{-\frac{1}{2\beta\eta^{2}\gamma^{2}}t}\right\|_{L^{1}_{t}}
\left\|\left(|\xi||\partial_{t}\hat{u}^{\gamma}|\!+\!|\xi|^{2}\left|\hat{G}_{4}^{\gamma}\right|+|\xi|^{3}|\hat{u}^{\gamma}|+
|\xi|\left|\hat{G}_{3}^{\gamma}\right|\right)\chi_{\{0\leqslant|\xi|\leqslant R\}}\right\|_{L^{2}_{t,\rm loc}L^{2}_{\xi}}\\
&\quad+\left\|\left\|e^{-\beta\eta^{2}|\xi|^{2}t}\chi_{[0,h]}(t)\right\|_{L^{1}_{t}}
\left\||\xi||\partial_{t}\hat{u}^{\gamma}|\!+\!|\xi|^{2}\left|\hat{G}_{4}^{\gamma}\right|\!+\!|\xi|^{3}|\hat{u}^{\gamma}|\!+
\!|\xi|\left|\hat{G}_{3}^{\gamma}\right|\right\|_{L^{2}_{t,\rm loc}}\chi_{\{0\leqslant|\xi|\leqslant R\}}\right\|_{L^{2}_{\xi}}\\
&\quad+\left\|e^{-\frac{1}{2\beta\eta^{2}\gamma^{2}}t}\chi_{[0,h]}(t)\right\|_{L^{1}_{t}}
\left\|\left(|\xi||\partial_{t}\hat{u}^{\gamma}|\!+\!|\xi|^{2}\left|\hat{G}_{4}^{\gamma}\right|\!+\!|\xi|^{3}|\hat{u}^{\gamma}|\!
+\!|\xi|\left|\hat{G}_{3}^{\gamma}\right|\right)\chi_{\{0\leqslant|\xi|\leqslant R\}}\right\|_{L^{2}_{t,\rm loc}L^{2}_{\xi}}\\
&\lesssim \left(\gamma^{2}+hR\right)\left(\left\|E^{0\gamma}\right\|_{L_{x}^{2}}+\left\|B^{0\gamma}\right\|_{L_{x}^{2}}\right)+\left(
h+\gamma^{2}\right)(R^{4}\left\|\partial_{t}u^{\gamma}\right\|_{L_{t,\rm loc}^{2}\dot{H}_{x}^{-3}}
+R^{2}\left\|u^{\gamma}\right\|_{L_{t,\rm loc}^{2}\dot{H}_{x}^{1}})\\
&\quad +\left(h+\gamma^{2}\right)\left(R^{2}\left\|\hat{G}_{4}^{\gamma}\chi_{\{0\leqslant|\xi|\leqslant R\}}\right\|_{L_{t,\rm loc}^{2}L_{\xi}^{2}}\!+\!R\left\|\hat{G}_{3}^{\gamma}\chi_{\{0\leqslant|\xi|\leqslant R\}}\right\|_{L_{t,\rm loc}^{2}L_{\xi}^{2}}\right).
\end{align*}}
\!\!By virtue of the uniform estimates of $G_{3}^{\gamma}$ and $G_{4}^{\gamma}$ in \eqref{wconvg3g4} and $\|\partial_{t}u^{\gamma}\|_{L_{t,\rm loc}^{2}\dot{H}_{x}^{-3}} $ in \eqref{partialtu}, and applying Plancherel's theorem, we obtain
$$\limsup\limits_{h\rightarrow 0}\sup_{\gamma}\left\|B^{\gamma}_{\ll}(t+h)-B^{\gamma}_{\ll}(t)\right\|_{L_{t,\rm loc}^{2}L_{\xi}^{2}}=0,$$
whence, utilizing Riesz-Fr\'{e}chet-Kolmogorov compactness criterion (for more details, see \cite{YK}), the relative compactness of $B^{\gamma}_{\ll}$ in the strong topology of $L^{2}_{t,x,\rm loc}$ can be deduced, that is,
\begin{equation}\label{convll}
B^{\gamma}_{\ll}\rightarrow B_{\ll}\;\;{\rm in} \;\;L^{2}_{t,x,\rm loc},
\end{equation}
where $B_{\ll}=\mathcal{F}^{-1}\left(\chi_{\{0\leqslant|\xi|\leqslant R\}}\hat{B}\right).$

Next, we handle with the term $B^{\gamma}_{<}$. Employing Lemma \ref{lemma3.2} and the estimates \eqref{esteb} and \eqref{es}, from \eqref{B2}, we find that
\begin{align*}
&\left|\hat{B}^{\gamma}_{<}\right|\\
&\lesssim \left(e^{-\beta\eta^{2}|\xi|^{2}t}+e^{-\frac{1}{2\beta\eta^{2}\gamma^{2}}t}\right)
\left(\left|\gamma b^{0\gamma}\right|+\left|\hat{B}^{0\gamma}\right|
+\left|\xi\times\hat{u}^{0\gamma}\right|\right)\chi_{\left\{R<|\xi|\leqslant\frac{1}{2K\beta\eta^{2}\gamma}\right\}}\\
&\quad+\gamma e^{-\frac{1}{2\beta\eta^{2}\gamma^{2}}t}\left|\hat{B}^{0\gamma}\right|
\!+\!\int_{0}^{t}\left[\frac{1}{\gamma^{2}}e^{-\frac{1}{2\beta\eta^{2}\gamma^{2}}(t-\tau)}+|\xi|^{2}e^{-\beta\eta^{2}|\xi|^{2}(t-\tau)}\right]
\left|\xi\times\hat{u}^{\gamma}\right|\chi_{\left\{R<|\xi|\leqslant\frac{1}{2K\beta\eta^{2}\gamma}\right\}}d\tau\\
&\quad+\int_{0}^{t}e^{-\beta\eta^{2}|\xi|^{2}(t-\tau)}
\left(|\xi|^{2}\left|\hat{G}_{4}^{\gamma}\right|\!+\!|\xi|^{3}\left|\hat{u}^{\gamma}\right|+
|\xi|\left|\hat{G}_{3}^{\gamma}\right|\right)\chi_{\left\{R<|\xi|\leqslant\frac{1}{2K\beta\eta^{2}\gamma}\right\}}d\tau.
\end{align*}
Then, one has
\begin{align*}
&\left\|\hat{B}^{\gamma}_{<}\right\|_{L_{t,\rm loc}^{2}L_{\xi}^{2}}\\
&\lesssim \left\|\left\|e^{-\beta\eta^{2}|\xi|^{2}t}\right\|_{L^{2}_{t}}
\left(\left|\hat{E}^{0\gamma}\right|+\left|\hat{B}^{0\gamma}\right|+\left|\xi\times \hat{u}^{0\gamma}\right|\right)
\chi_{\left\{R<|\xi|\leqslant\frac{1}{2K\beta\eta^{2}\gamma}\right\}}\right\|_{L^{2}_{\xi}}\\
&\quad+\left\||\xi|^{2}\left\|e^{-\beta\eta^{2}|\xi|^{2}t}\right\|_{L^{1}_{t}}\left\|\xi\times \hat{u}^{\gamma}\right\|_{L^{2}_{t,\rm loc}}\chi_{\left\{R<|\xi|\leqslant\frac{1}{2K\beta\eta^{2}\gamma}\right\}}\right\|_{L^{2}_{\xi}}\\
&\quad+\frac{1}{\gamma^{2}}\left\|e^{-\frac{t}{2\beta\eta^{2}\gamma^{2}}}\right\|_{L^{1}_{t}}\left\|\xi\times \hat{u}^{\gamma}
\chi_{\left\{R<|\xi|\leqslant\frac{1}{2K\beta\eta^{2}\gamma}\right\}}\right\|_{L^{2}_{t,\rm loc}L^{2}_{\xi}}\\
&\quad+\left\|\left\|e^{-\beta\eta^{2}|\xi|^{2}t}\right\|_{L^{1}_{t}}
\left\||\xi|^{2}\left|\hat{G}_{4}^{\gamma}\right|+|\xi|^{3}\left|\hat{u}^{\gamma}\right|\!+\!|\xi|\left|\hat{G}_{3}^{\gamma}\right|\right\|_{L^{2}_{t,\rm loc}}
\chi_{\left\{R<|\xi|\leqslant\frac{1}{2K\beta\eta^{2}\gamma}\right\}}\right\|_{L^{2}_{\xi}}\\
&\lesssim\frac{1}{R}\left(\left\|E^{0\gamma}\right\|_{L_{x}^{2}}+\left\|B^{0\gamma}\right\|_{L_{x}^{2}}\right)
+\left\|\hat{u}^{0\gamma}\chi_{\left\{R<|\xi|\leqslant \frac{1}{2K\beta\eta^{2}\gamma}\right\}}\right\|_{L_{\xi}^{2}}
+\left\|\xi\times\hat{u}^{\gamma}\chi_{\left\{R<|\xi|\leqslant \frac{1}{2K\beta\eta^{2}\gamma}\right\}}\right\|_{L_{t,\rm loc}^{2}L_{\xi}^{2}}\\
&\quad+\left\|\hat{G}_{4}^{\gamma}\chi_{\left\{R<|\xi|\leqslant \frac{1}{2K\beta\eta^{2}\gamma}\right\}}\right\|_{L_{t,\rm loc}^{2}L_{\xi}^{2}}+\left\||\xi|^{-1}\hat{G}_{3}^{\gamma}\chi_{\left\{R<|\xi|\leqslant \frac{1}{2K\beta\eta^{2}\gamma}\right\}}\right\|_{L_{t,\rm loc}^{2}L_{\xi}^{2}}\\
&\lesssim\frac{1}{R}\left(\left\|E^{0\gamma}\right\|_{L_{x}^{2}}+\left\|B^{0\gamma}\right\|_{L_{x}^{2}}\right)
+\frac{1}{R^{s}}\left(\left\|u^{0\gamma}\right\|_{\dot{H}^{s}}+\left\|u^{\gamma}\right\|_{L^{2}_{t,\rm loc}\dot{H}^{1+s}}\right)\\
&\quad+R^{\frac{1}{2}-s}\left\|\nabla G_{4}^{\gamma}\right\|_{L^{2}_{t,\rm loc}\dot{H}^{s-\frac{3}{2}}}
+\frac{1}{R}\left\|G_{3}^{\gamma}\right\|_{L^{2}_{t,\rm loc}L^{2}_{\xi}}\\
&\lesssim\frac{1}{R}\left(\left\|E^{0\gamma}\right\|_{L_{x}^{2}}+\left\|B^{0\gamma}\right\|_{L_{x}^{2}}\right)
+\frac{1}{R^{s}}\left(\left\|u^{0\gamma}\right\|_{\dot{H}^{s}}+\left\|u^{\gamma}\right\|_{L^{2}_{t,\rm loc}\dot{H}^{1+s}}\right)\\
&\quad+R^{\frac{1}{2}-s}\left\|\nabla G_{4}^{\gamma}\right\|_{L^{2}_{t,\rm loc}L^{\frac{3}{3-s}}}+\frac{1}{R}\left\|G_{3}^{\gamma}\right\|_{L^{2}_{t,\rm loc}L^{2}_{\xi}}.
\end{align*}
In view of the uniform bounds \eqref{bddg3g4} of $G^{\gamma}_{3}$ and $G^{\gamma}_{4}$, we conclude that
\begin{equation}\label{convl}
\limsup\limits_{\gamma\rightarrow 0}\left\|B^{\gamma}_{<}\right\|_{L_{t,\rm loc}^{2}L_{\xi}^{2}}\lesssim R^{\frac{1}{2}-s}.
\end{equation}

Thirdly, we deal with the term $B_{\sim}^{\gamma}$. Clearly, \eqref{B2} can be recast as
\begin{align*}
\hat{B}^{\gamma}&=\left[\gamma e^{t\lambda_{-}}+(1-\gamma)e^{t\lambda_{+}}\right]B^{0\gamma}+\gamma\lambda_{-}te^{t\lambda_{+}}\frac{1-e^{t(\lambda_{-}-\lambda_{+})}}{t(\lambda_{-}-\lambda_{+})}sb^{0\gamma}\\
&\quad+(1-\gamma)\lambda_{+}te^{t\lambda_{+}}\frac{1-e^{t(\lambda_{-}-\lambda_{+})}}{t(\lambda_{-}-\lambda_{+})}B^{0\gamma}
-\frac{1}{\beta\eta\gamma^{2}}te^{t\lambda_{+}}\frac{1-e^{t(\lambda_{-}-\lambda_{+})}}{t(\lambda_{-}-\lambda_{+})}\xi\times\hat{u}^{0\gamma}\\
&\quad-\frac{i}{\beta\eta\gamma^{2}}\int_{0}^{t}e^{(t-\tau)\lambda_{-}}\xi\times\hat{u}^{\gamma}
+\lambda_{+}(t-\tau)e^{(t-\tau)\lambda_{+}}\frac{e^{(t-\tau)(\lambda_{-}-\lambda_{+})}-1}{(t-\tau)(\lambda_{-}-\lambda_{+})}\xi\times\hat{u}^{\gamma}d\tau\\
&\quad+\frac{1}{\beta\eta\gamma^{2}}\int_{0}^{t}(t-\tau)e^{(t-\tau)\lambda_{+}}\frac{1-e^{(t-\tau)(\lambda_{-}-\lambda_{+})}}{(t-\tau)(\lambda_{-}-\lambda_{+})}
\xi\times\left(i\xi\cdot\hat{G}_{4}^{\gamma}+|\xi|^{2}\hat{u}^{\gamma}-\frac{1}{\eta}\hat{G}_{3}^{\gamma}\right)d\tau.
\end{align*}
For fixed $1<K<\frac{\sqrt{5}}{2}$, using the complex mean value theorem (see \cite{PDM}), we can infer that
$$ \left|e^{t\lambda_{+}}\frac{1-e^{t(\lambda_{-}-\lambda_{+})}}{t(\lambda_{-}-\lambda_{+})}\right|\leq e^{-\frac{\omega}{\gamma^{2}}t}$$
for some constant $\omega>0$ depending on $K$, provided $\frac{1}{2K\beta\eta^{2}\gamma}<|\xi|\leqslant\frac{K}{2\beta\eta^{2}\gamma}$. Combining Lemma \ref{lemma3.2}, estimates \eqref{esteb} and \eqref{es}, we have
\begin{align*}
\left|\hat{B}^{\gamma}_{\sim}\right|&\lesssim\left(\gamma e^{-\frac{1}{2\beta\eta^{2}\gamma^{2}}t}+
e^{\frac{-K+\sqrt{K^{2}-1}}{2K\beta\eta^{2}\gamma^{2}}t}+\frac{1}{\gamma^{2}}te^{-\frac{\omega}{\gamma^{2}}t}\right)\left|\hat{B}^{0\gamma}\right|
+\frac{1}{\gamma}te^{-\frac{\omega}{\gamma^{2}}t}\left(\left|\hat{B}^{0\gamma}\right|+\left|\hat{E}^{0\gamma}\right|\right)\\
&\quad+\frac{1}{\gamma^{2}}te^{-\frac{\omega}{\gamma^{2}}t}\left|\xi\times\hat{u}^{0\gamma}\right|
\chi_{\left\{\frac{1}{2K\beta\eta^{2}\gamma}<|\xi|\leqslant\frac{K}{2\beta\eta^{2}\gamma}\right\}}\\
&\quad+\frac{1}{\gamma^{2}}\int_{0}^{t}\left[e^{-\frac{1}{2\beta\eta^{2}\gamma^{2}}(t-\tau)}+
\frac{1}{\gamma^{2}}(t-\tau)e^{-\frac{\omega}{\gamma^{2}}(t-\tau)}\right]\left|\xi\times\hat{u}^{\gamma}\right|
\chi_{\left\{\frac{1}{2K\beta\eta^{2}\gamma}<|\xi|\leqslant\frac{K}{2\beta\eta^{2}\gamma}\right\}}d\tau\\
&\quad+\frac{1}{\gamma^{2}}\int_{0}^{t}(t-\tau)e^{-\frac{\omega}{\gamma^{2}}(t-\tau)}
\left(|\xi|^{2}\left|\hat{G}_{4}^{\gamma}\right|+|\xi|^{3}\left|\hat{u}^{\gamma}\right|+|\xi|\left|\hat{G}_{3}^{\gamma}\right|\right)
\chi_{\left\{\frac{1}{2K\beta\eta^{2}\gamma}<|\xi|\leqslant\frac{K}{2\beta\eta^{2}\gamma}\right\}}d\tau.
\end{align*}
Thus, we get
\begin{align*}
&\left\|\hat{B}^{\gamma}_{\sim}\right\|_{L_{t,\rm loc}^{2}L_{\xi}^{2}}\\
&\lesssim \left( \left\| e^{\frac{-K+\sqrt{K^{2}-1}}{2K\beta\eta^{2}\gamma^{2}}t}\right\| _{L^{2}_{t}}+\frac{1}{\gamma^{2}}\left\|te^{-\frac{\omega}{\gamma^{2}}t}\right\|_{L^{2}_{t}}\right)
\left(\left\|E^{0\gamma}\right\|_{L_{x}^{2}}+\left\|B^{0\gamma}\right\|_{L_{x}^{2}}\right)\\
&\quad+\frac{1}{\gamma^{2}}\left\|te^{-\frac{\omega}{\gamma^{2}}t}\right\|_{L^{2}_{t}}\left\|\xi\times \hat{u}^{0\gamma}\chi_{\left\{\frac{1}{2K\beta\eta^{2}\gamma}<|\xi|\leqslant\frac{K}{2\beta\eta^{2}\gamma}\right\}}\right\|_{L_{\xi}^{2}}\\
&\quad+\frac{1}{\gamma^{2}}\left\|e^{-\frac{1}{2\beta\eta^{2}\gamma^{2}}t}\right\|_{L^{1}_{t}}\left\|\xi\times \hat{u}^{\gamma}\chi_{\left\{\frac{1}{2K\beta\eta^{2}\gamma}<|\xi|\leqslant\frac{K}{2\beta\eta^{2}\gamma}\right\}}\right\|_{L^{2}_{t,\rm loc}L_{\xi}^{2}}\\
&\quad+\frac{1}{\gamma^{2}}\left\|te^{-\frac{\omega}{\gamma^{2}}t}\right\|_{L^{1}_{t}}
\left\|\left(|\xi|^{2}\left|\hat{G}_{4}^{\gamma}\right|+|\xi|^{3}|\hat{u}^{\gamma}|+|\xi|\left|\hat{G}_{3}^{\gamma}\right|\right)
\chi_{\left\{\frac{1}{2K\beta\eta^{2}\gamma}<|\xi|\leqslant\frac{K}{2\beta\eta^{2}\gamma}\right\}}\right\|_{L^{2}_{t,\rm loc}L_{\xi}^{2}}\\
&\lesssim \gamma\left(\left\|E^{0\gamma}\right\|_{L_{x}^{2}}+\left\|B^{0\gamma}\right\|_{L_{x}^{2}}\right)+\gamma^{s}\left(\left\|u^{0\gamma}\right\|_{\dot{H}_{x}^{s}}
+\left\|u^{\gamma}\right\|_{L_{t,\rm loc}^{2}\dot{H}_{x}^{1+s}}\right)\\
&\quad+\gamma^{s-\frac{1}{2}}\left\|\nabla G_{4}^{\gamma}\right\|_{L^{2}_{t,\rm loc}\dot{H}^{s-\frac{3}{2}}}+\gamma\left\|G_{3}^{\gamma}\right\|_{L^{2}_{t,\rm loc}L^{2}_{\xi}}\\
&\lesssim \gamma\left(\left\|E^{0\gamma}\right\|_{L_{x}^{2}}+\left\|B^{0\gamma}\right\|_{L_{x}^{2}}\right)+\gamma^{s}\left(\left\|u^{0\gamma}\right\|_{\dot{H}_{x}^{s}}
+\left\|u^{\gamma}\right\|_{L_{t,\rm loc}^{2}\dot{H}_{x}^{1+s}}\right)\\
&\quad+\gamma^{s-\frac{1}{2}}\left\|\nabla G_{4}^{\gamma}\right\|_{L^{2}_{t,\rm loc}L^{\frac{3}{3-s}}}+\gamma\left\|G_{3}^{\gamma}\right\|_{L^{2}_{t,\rm loc}L^{2}_{\xi}}.
\end{align*}
By virtue of the uniform bounds \eqref{bddg3g4} of $G^{\gamma}_{3}$ and $G^{\gamma}_{4}$, we deduce that
\begin{equation}\label{convmid}
\limsup\limits_{\gamma\rightarrow 0}\left\|B^{\gamma}_{\sim}\right\|_{L_{t,\rm loc}^{2}L_{\xi}^{2}}=0.
\end{equation}

Finally, we deal with the term $B^{\gamma}_{>}$. Combining Lemma \ref{lemma3.2}, estimates \eqref{esteb} and \eqref{es}, it follows from \eqref{B2} that
\begin{align*}
\left|\hat{B}^{\gamma}_{>}\right|
&\lesssim e^{-\frac{1}{2\beta\eta^{2}\gamma^{2}}t}\left(\left|\hat{B}^{0\gamma}\right|+\left|\hat{E}^{0\gamma}\right|\right)
+\frac{1}{\gamma}e^{-\frac{1}{2\beta\eta^{2}\gamma^{2}}t}\left|\hat{u}^{0\gamma}\right|
\chi_{\left\{\frac{K}{2\beta\eta^{2}\gamma}<|\xi|\leqslant\phi\left(\frac{\gamma}{\delta}\right)\right\}}\\
&\quad+\frac{1}{\gamma^{2}}\int_{0}^{t}e^{-\frac{1}{2\beta\eta^{2}\gamma^{2}}(t-\tau)}\left|\xi\times\hat{u}^{\gamma}\right|
\chi_{\left\{\frac{K}{2\beta\eta^{2}\gamma}<|\xi|\leqslant\phi\left(\frac{\gamma}{\delta}\right)\right\}}d\tau\\
&\quad+\frac{1}{\gamma|\xi|}\int_{0}^{t}e^{-\frac{1}{2\beta\eta^{2}\gamma^{2}}(t-\tau)}
\left(|\xi|^{2}\left|\hat{G}_{4}^{\gamma}\right|+|\xi|^{3}\left|\hat{u}^{\gamma}\right|+|\xi|\left|\hat{G}_{3}^{\gamma}\right|\right)
\chi_{\left\{\frac{K}{2\beta\eta^{2}\gamma}<|\xi|\leqslant\phi\left(\frac{\gamma}{\delta}\right)\right\}}d\tau.
\end{align*}
Therefore, we have
\begin{align*}
&\left\|\hat{B}^{\gamma}_{>}\right\|_{L_{t,\rm loc}^{2}L_{\xi}^{2}}\\
&\lesssim \left\|e^{-\frac{1}{2\beta\eta^{2}\gamma^{2}}t}\right\|_{L^{2}_{t}}\left(\left\|E^{0\gamma}\right\|_{L_{x}^{2}}+\left\|B^{0\gamma}\right\|_{L_{x}^{2}}
+\frac{1}{\gamma}\left\|\hat{u}^{0\gamma}\chi_{\left\{\frac{K}{2\beta\eta^{2}\gamma}<|\xi|\leqslant\phi\left(\frac{\gamma}{\delta}\right)\right\}
}\right\|_{L_{\xi}^{2}}\right)\\
&\quad+\frac{1}{\gamma^{2}}\left\|e^{-\frac{1}{2\beta\eta^{2}\gamma^{2}}t}\right\|_{L^{1}_{t}}\left\|\xi\times \hat{u}^{\gamma}\chi_{\left\{\frac{K}{2\beta\eta^{2}\gamma}<|\xi|\leqslant\phi\left(\frac{\gamma}{\delta}\right)\right\}}\right\|_{L^{2}_{t,\rm loc}L_{\xi}^{2}}\\
&\quad+\frac{1}{\gamma}\left\|e^{-\frac{1}{2\beta\eta^{2}\gamma^{2}}t}\right\|_{L^{1}_{t}}\left\|
\left(|\xi|\left|\hat{G}_{4}^{\gamma}\right|+|\xi|^{2}|\hat{u}^{\gamma}|+\left|\hat{G}_{3}^{\gamma}\right|\right)
\chi_{\left\{\frac{K}{2\beta\eta^{2}\gamma}<|\xi|\leqslant\phi\left(\frac{\gamma}{\delta}\right)\right\}}\right\|_{L^{2}_{t,\rm loc}L_{\xi}^{2}}\\
&\lesssim \gamma\left(\left\|E^{0\gamma}\right\|_{L_{x}^{2}}+\left\|B^{0\gamma}\right\|_{L_{x}^{2}}\right)+\gamma^{s}\left(\left\|u^{0\gamma}\right\|_{\dot{H}_{x}^{s}}
+\left\|u^{\gamma}\right\|_{L_{t,\rm loc}^{2}\dot{H}_{x}^{1+s}}\right)\\
&\quad+\gamma\left(\phi(\frac{\gamma}{\delta})^{\frac{3}{2}-s}\left\|\nabla G_{4}^{\gamma}\right\|_{L^{2}_{t,\rm loc}\dot{H}^{s-\frac{3}{2}}}+\phi(\frac{\gamma}{\delta})^{1-s}\left\|u^{\gamma}\right\|_{L_{t,\rm loc}^{2}\dot{H}_{x}^{1+s}}+
\left\|G_{3}^{\gamma}\right\|_{L^{2}_{t,\rm loc}L^{2}_{x}}\right)\\
&\lesssim  \gamma\left(\left\|E^{0\gamma}\right\|_{L_{x}^{2}}+\left\|B^{0\gamma}\right\|_{L_{x}^{2}}\right)+\gamma^{s}\left(\left\|u^{0\gamma}\right\|_{\dot{H}_{x}^{s}}
+\left\|u^{\gamma}\right\|_{L_{t,\rm loc}^{2}\dot{H}_{x}^{1+s}}\right)\\
&\quad+\delta\left\|\nabla G_{4}^{\gamma}\right\|_{L^{2}_{t,\rm loc}L^{\frac{3}{3-s}}}+\gamma^{\frac{1}{3-2s}}\left\|u^{\gamma}\right\|_{L_{t,\rm loc}^{2}\dot{H}_{x}^{1+s}}
+\gamma\left\|G_{3}^{\gamma}\right\|_{L^{2}_{t,\rm loc}L^{2}_{x}},
\end{align*}
whence, in view of the uniform bounds \eqref{bddg3g4} of $G^{\gamma}_{3}$ and $G^{\gamma}_{4}$, one has
\begin{equation}\label{convh}
\limsup\limits_{\gamma\rightarrow 0}\left\|B^{\gamma}_{<}\right\|_{L_{t,\rm loc}^{2}L_{\xi}^{2}}\lesssim \delta.
\end{equation}

Now, we are ready to prove the weak stability of the Lorentz force $j^{\gamma}\times B^{\gamma}$.
Combining the estimates \eqref{convhh}, \eqref{convll}, \eqref{convl}, \eqref{convmid}, \eqref{convh} and the fact that $j^\gamma \rightharpoonup j$ in $L^2_{
t,x}$, we can deduce from \eqref{decompose} that
\begin{align*}&\limsup\limits_{\gamma\rightarrow 0}\left|\int_{\mathbb{R}^{+}\times \mathbb{R}^{3}}\left(j^{\gamma}\times B^{\gamma}-j\times B\right)\cdot\varphi\;dtdx\right|\\
&\lesssim \int_{\mathbb{R}^{+}\times \mathbb{R}^{3}}j\times\left( B_{\ll}-B\right)\cdot\varphi\;dtdx+\delta+R^{\frac{1}{2}-s}
\end{align*}
Notice that as $R\rightarrow\infty$,
$$B_{\ll}\rightarrow B \;\;{\rm in} \;\;L_{t,\rm loc}^{2}L_{x}^{2},$$
by the arbitrariness of $R>0$ (large) and $\delta>0$ (small), we have
$$\lim\limits_{\gamma\rightarrow 0}\int_{\mathbb{R}^{+}\times \mathbb{R}^{3}}j^{\gamma}\times B^{\gamma}\cdot\varphi\;dtdx=
\int_{\mathbb{R}^{+}\times \mathbb{R}^{3}}j\times B\cdot\varphi\;dtdx,$$
which concludes the desired results of Theorem \ref{Mian}. \qed

\section*{Acknowledgments}
H. Wang was supported by the National Natural Science Foundation of China (No. 11901066), the
Natural Science Foundation of Chongqing (No. cstc2019jcyj-msxmX0167) and Project No. 2019CDXYST0015,  No. 2020CDJQY-A040 supported by the Fundamental Research Funds for the Central Universities. Q. Xu's research was supported by the National Natural Science Foundation of China (No. 11901070), the Natural Science Foundation of Chongqing (No. cstc2018jcyjAX0010) and  the Open Project of Key Laboratory No.CSSXKFKTZ202005, School of Mathematical Sciences, Chongqing Normal University.


\begin{thebibliography}{aa}
\footnotesize

\bibitem{ADFL}
M. Acheritogaray, P. Degond, A. Frouvelle, J.G. Liu, Kinetic formulation and global existence for the Hall-magnetohydrodynamic system, Kinet. Relat. Models, 4 (2011) 901-918.

\bibitem{AG}
D. Ars\'{e}nio, I. Gallagher, Solutions of Navier-Stokes-Maxwell systems in large energy spaces, Trans. Amer. Math. Soc., 373 (2020) 3853-3884.


\bibitem{AIM}
D. Ars\'{e}nio, S. Ibrahim, N. Masmoudi, A derivation of the magnetohydrodynamic system from Navier-Stokes-Maxwell systems, Arch. Rational Mech. Anal., 216 (2015) 767-812.

\bibitem{BT}
S.A. Balbus, C. Terquem, Linear analysis of the Hall effect in protostellar disks, Astrophys. J., 552 (2001) 235-247.

\bibitem{BF}
M.J. Benvenutti, L.C.F. Ferreira, Existence and stability of global large strong solutions
for the Hall-MHD system, Differ. Integral Equ., 29 (2016), 9771000.

\bibitem{CLMBC}
L.M.B.C. Campos, On hydromagnetic waves in atmospheres with application to the sun, Theor. Comput. Fluid Dyn., 10 (1998) 37-70.

\bibitem{CDL}
D. Chae, P. Degond, J.G. Liu, Well-posedness for Hall-magnetohydrodynamics, Ann. Inst. Henri Poincar\'{e}, Anal. Non Lin\'{e}aire, 31(3) (2014) 555-565.

\bibitem{CL}
D. Chae, J. Lee, On the blow-up criterion and small data global existence for the Hall-magneto-hydrodynamics, J. Differential Equations, 256 (2014) 3835-3858.

\bibitem{CS}
D. Chae, M. Schonbek, On the temporal decay for the Hall-magnetohydrodynamic equations, J. Differential Equations, 255 (2013) 3971-3982.

\bibitem{CWW}
D. Chae, R. Wan, J. Wu, Local well-posedness for the Hall-MHD equations with fractional magnetic diffusion, arXiv:1404.0486v2, 2014.

\bibitem{CW}
D. Chaea, S. Weng, Singularity formation for the incompressible Hall-MHD equations without resistivity, Ann. Inst. Henri Poincar\'{e}, Anal. Non Lin\'{e}aire, 33 (2016) 1009-1022.

\bibitem{CW}
D. Chae, J. Wolf, On partial regularity for the 3D non-stationary Hall magnetohydrodynamics equations on the plane, Comm. Math. Phys.,  354 (2017) 213-230.

\bibitem{DM}
M. Dai, Regularity criterion for the 3D Hall-magneto-hydrodynamics, J. Differential Equations, 261 (2016) 573-591.

\bibitem{DM1}
M. Dai, Non-uniqueness of Leray-Hopf weak solutions of the 3D Hall-MHD system, arXiv:1812.11311.

\bibitem{DH1}
M. Dai and H. Liu. Long time behavior of solutions to the 3D Hall-magneto-hydrodynamics
system with one diffusion,  J. Differential Equations, 266 (2019) 7658-7677.

\bibitem{DH}
M. Dai, H. Liu, On well-posedness of generalized Hall-magneto-hydrodynamics, arXiv:1906.02284v3.

\bibitem{DRG}
J. Dreher, V. Runban, R. Grauer, Axisymmetric flows in Hall-MHD: a tendency towards finite-time singularity formation, Phys. Scr., 72 (2005) 451-455.

\bibitem{DS}
E. Dumas, F. Sueur, On the weak solutions to the Maxwell-Landau-Lifshitz equations and to the Hall-magnetohydrodynamic equations, Comm. Math. Phys., 330 (2014) 1179-1225.

\bibitem{FHN}
J. Fan, S. Huang, G. Nakamura, Well-posedness for the axisymmetric incompressible viscous Hall-magnetohydrodynamic equations, Appl. Math. Lett., 26(9) (2013) 963-967.

\bibitem{FSZ}
J. Fan, B. Samet, Y. Zhou, A regularity criterion for a generalized Hall-MHD system, Comput. Math. Appl., 74 (2017) 2438-2443.

\bibitem{FTG}
T.G. Forbes, Magnetic reconnection in solar flares, Geophys. Astrophys. Fluid Dyn., 62 (1991) 15-36.

\bibitem{GI}
P. Germain, S. Ibrahim, N. Masmoudi, Well-posedness of the Navier-Stokes-Maxwell equations, Proc. Roy. Soc. Edinburgh Sect. A, 144 (2014) 71-86.

\bibitem{GISY}
Y. Giga, S. Ibrahim, S. Shen, T. Yoneda,  Global well-posedness for a two-fluid model, arXiv:1411.0917.

\bibitem{HG}
H. Homann, R. Grauer, Bifurcation analysis of magnetic reconnection in Hall-MHD systems, Phys. D, 208 (2005) 59-72.

\bibitem{IK}
S. Ibrahim, S. Keraani, Global small solutions for the Navier-Stokes-Maxwell system, SIAM J. Math. Anal., 43 (2011) 2275-2295.

\bibitem{IY}
S. Ibrahim, T. Yoneda, Local solvability and loss of smoothness of the Navier-Stokes-Maxwell equations with large initial data, J. Math. Anal. Appl., 396 (2012) 555-561.

\bibitem{JZ}
Z. Jiang, M. Zhu, Regularity criteria for the 3D generalized MHD and Hall-MHD systems, Bull. Malays. Math. Sci. Soc., 41 (2018) 105-122.

\bibitem{KL}
M. Kwak, B. Lkhagvasuren, Global wellposedness for Hall-MHD equations, Nonlinear Anal., 174 (2018) 104-117.

\bibitem{LMJ}
M.J. Lighthill, Studies on magnetohydrodynamic waves and other anisotropic wave motions, Philos. Trans. R. Soc. Lond. Ser. A, (1960) 397-430.

\bibitem{MN}
N. Masmoudi, Global well-posedness for the Maxwell-Navier-Stokes system in 2D, J. Math. Pures Appl., 93 (2010) 559-571.

\bibitem{MGM}
P.D. Mininni, D.O. G\`{o}mez, S.M. Mahajan, Dynamo action in magnetohydrodynamics and Hall magnetohydrodynamics, Astrophys. J., 587 (2003) 472-481.

\bibitem{PZ}
N. Pan, M. Zhu, A new regularity criterion for the 3D generalized Hall-MHD system with $\beta\in(\frac{1}{2},1]$, J. Math. Anal. Appl., 445 (2017) 604-611.

\bibitem{PDM}
J.P. Pemba, A.R. Davies, N.K. Muoneke, A Complexification of Rolle¡¯s Theorem, Applications and Applied Mathematics: An International Journal(AAM), 2 (2007) 28-31.

\bibitem{PM}
J.M. Polygiannakis, X. Moussas, A review of magneto-vorticity induction in Hall-MHD plasmas, Plasma Phys. Control. Fusion, 43 (2001) 195-221.

\bibitem{SJ}
J. Simon, Compact sets in the space $L^{ p}(0,T; B)$, Annali di Matematica Pura ed Applicata, 146 (1986) 65-96.

\bibitem{SS}
B. Srinivasan, U. Shumlak, Analytical and computational study of the ideal full two-fluid plasma model and asymptotic approximations for Hall-magnetohydrodynamics, Physics of Plasmas, 18(9) (2011) 620.

\bibitem{WR}
R. Wan, Global regularity for generalized Hall Magneto-Hydrodynamics systems, Electron. J. Differ. Equ., 179 (2015) 1-18.

\bibitem{WZ1}
R. Wan, Y. Zhou, Low Regularity Well-Posedness for the 3D Generalized Hall-MHD System, Acta Appl. Math., 147 (2017) 95-111.

\bibitem{WZ}
R. Wan, Y. Zhou, On global existence, energy decay and blow-up criteria for the Hall-MHD system, J. Differential Equations, 259 (2015) 5982-6008.

\bibitem{WM}
M. Wardle, Star formation and the Hall effect, Astrophys. Space Sci., 292 (2004) 317-323.

\bibitem{WYT}
X. Wu, Y. Yu, Y. Tang, Global existence and asymptotic behavior for the 3D generalized Hall-MHD system, Nonlinear Anal., 151 (2017) 41-50.

\bibitem{YZ}
Z. Ye, Regularity criteria and small data global existence to the generalized viscous Hall-magneto-hydrodynamics, Comput. Math. Appl., 70 (2015) 2137-2154.

\bibitem{YK}
K. Yosida, Functional analysis. Classics in mathematics, Springer, Berlin, 1995.

\end{thebibliography}
\end{document}